\documentclass[10pt]{article}
\usepackage{amsmath,amssymb,amsfonts,theorem}
\usepackage{graphicx,epsfig,color,psfrag} %
\usepackage{subfig}
\graphicspath{{./Figures/}}

\newcommand{\until}[1]{\{1,\dots, #1\}}
\newcommand{\subscr}[2]{#1_{\textup{#2}}}
\newcommand{\supscr}[2]{#1^{\textup{#2}}}
\newcommand{\setdef}[2]{\{#1 \, | \; #2\}}
\newcommand{\map}[3]{#1: #2 \rightarrow #3}
\newcommand{\union}{\operatorname{\cup}}
\newcommand{\intersection}{\ensuremath{\operatorname{\cap}}}
\newcommand{\intersect}{\intersection}

\newcommand{\bigintersect}{\bigcap}
\newcommand{\eps}{\varepsilon}
\newcommand{\degs}{^\circ\!}

\newcommand{\real}{\ensuremath{\mathbb{R}}}

\newcommand{\realnonnegative}{\ensuremath{\mathbb{R}}_{\geq0}}
\renewcommand{\natural}{{\mathbb{N}}}

\newtheorem{theorem}{Theorem}[section]
\newtheorem{proposition}[theorem]{Proposition}

{\theorembodyfont{\rmfamily} % These ones have roman font in their body
\newtheorem{remark}[theorem]{Remark}

}

\newcommand{\bl}{\textup{BL}}
\newcommand{\tu}{\textup{TU}}
\newcommand{\vc}{\supscr{v}{a}}
\newcommand{\vr}{\supscr{v}{r}}

\newcommand{\rreg}[2]{\mathcal{R}_{#1}^{#2}}
\newcommand{\creg}[2]{\mathcal{A}_{#1}^{#2}}

\newcommand{\vmax}{\subscr{v}{max}}
\newcommand{\card}[1]{{|#1|}}

\newcommand{\co}{\operatorname{\overline{co}}}
\newcommand{\Lci}{L}
\newcommand{\ai}{\textup{AI}}
\newcommand{\angleTol}{\subscr{\eps}{angle}}
\newcommand{\alphac}{\subscr{\alpha}{a}}
\newcommand{\alphar}{\subscr{\alpha}{r}}
\newcommand{\Rsr}{\subscr{R}{sr}}

\newcommand{\fc}{\subscr{f}{a}}
\newcommand{\fr}{\subscr{f}{r}}
\newcommand{\Fc}{\subscr{F}{a}}
\newcommand{\Fr}{\subscr{F}{r}}

\newcommand{\V}{\mathcal{V}} % un insieme di indici
\newcommand{\K}{\mathcal{K}} % filippov operator
\newcommand{\closest}[1]{{\textup{closest}_{#1}}}
\newcommand{\CoincSet}{S}%\subscr{S}{N}}
\newcommand{\DiscSet}{D}

\title{Effects of Anisotropic Interactions on the Structure of Animal Groups}

\author{Emiliano Cristiani\thanks{CEMSAC, Universit\`a di Salerno, Italy and IAC-CNR, Rome, Italy. E-mail: \texttt{emiliano.cristiani@gmail.com}}
\and Paolo Frasca\thanks{DIIMA, Universit\`a di Salerno, Italy and IAC-CNR, Rome, Italy. E-mail: \texttt{paolo.frasca@gmail.com}}
\and Benedetto Piccoli\thanks{IAC-CNR, Rome, Italy. E-mail: \texttt{b.piccoli@iac.cnr.it}
 Mail: IAC-CNR c/o Dipartimento di Matematica dell'Universit\`a di Roma Tor Vergata, Via della Ricerca Scientifica, 00133, Rome, Italy}
 }

\begin{document}
\maketitle

\begin{abstract}
This paper proposes an agent-based model which reproduces different structures of animal groups. The shape and structure of the group is the effect of simple interaction rules among individuals: each animal deploys itself depending on the position of a limited number of close group mates. The proposed model is shown to produce clustered formations, as well as lines and V-like formations. The key factors which trigger the onset of different patterns are argued to be the relative strength of attraction and repulsion forces and, most important, the anisotropy in their application.
\end{abstract}
\vskip0.5cm
\noindent\textbf{M.S.C.}
92D50, 92B05.
\vskip0.5cm

%%%%%%%%%%%%%%%%%%%%%%%%%%%%%%%%%%%%%%%%%%%%%%%%%%%%%%%%%%%%%%%%%%%%%%%%%%%%%%%%%%%%%%%%%%%%%%%%%%%%%%%%%%%
\section{Introduction}\label{sec:intro}
Group behavior in animals has greatly interested scientists and researchers in the past, and has received further attention in the last decades as a test case of self-organization. Recently it has attracted attention not only from biologists and ethologists, but also from physicists, mathematicians, and engineers. This interest has produced a huge amount of literature, which is well documented and reviewed. See, for instance, \cite{JK-GDR:02,DS:06,IG:08} and~\cite{FB-JC-SM:09}. Loosely speaking, the basic idea behind these works is that complex collective behavior arises from simple interactions among close animals. Following this idea, the aim of this paper is to investigate a simple model of interactions within groups which is able to reproduce rather different patterns and structures.

Let us briefly introduce the main features of our model and its relationships with literature. A detailed description is postponed to Section~\ref{sec:model}. First, our model is {\em agent-based}, in the sense that each animal is singularly considered. Furthermore it is {\em leaderless}, meaning that all animals act following the same set of rules and their behavior is not imposed by others. These assumptions are widely considered in literature, and accepted as biologically suitable for a variety of species. See, among others,~\cite{KW-JL:91,AH-CW:92,JKP-SVV-DG:02,IC-JK-RJ-GR-NF:02,IDC-JK-NRF-SAL:05}.

Second, our model is purely based on {\em attraction-repulsion} interactions between group mates. Attraction allows the group to be formed and stay tight, while repulsion allows to avoid collisions between group mates and keeps them well spaced. 
%This kind of interaction has been widely considered in the literature. 
In the majority of papers, attraction and repulsion are combined with velocity alignment. Here we keep aside the issue of alignment which has received a considerable attention in itself~\cite{TV-AC-EBJ-IC-OS:95,FC-SS:07}, and focus on (the superposition of) attraction and repulsion. Our approach is close to~\cite{KW-JL:91,AM-LE-LB-AS:03} in this respect. %In our model decision rules are not hierarchical, rather the ''forces'' act always together.

Third, each animal interacts with a {\em limited number} of group mates.
The idea of having a limited number of interacting neighbors is not new. The work~\cite{WDH:71} already considers attraction towards {\em the} nearest neighbor, while later experimental investigations found interaction with the closest two-four individuals \cite{IA:80}. This fact has been included in several models, among others~\cite{KW-JL:91,AH-CW:92,YI-KK:02,JKP-SVV-DG:02,RL-YXL-LE:09}, but it has not always been included in recent models \cite{SG-SAL-DIR:96,IC-JK-RJ-GR-NF:02,HK-CKH:03,IDC-JK-NRF-SAL:05}, in favor of a purely metric notion of neighborhood: interactions occur among group mates which are less than a threshold apart.
Very recently, the former idea has again been brought to attention by~\cite{AC-IG:08a}, where the authors present experimental results regarding fairly large flocks of starlings. Results show that interaction occurs with up to six-seven neighbors, no matter how far they are. The same paper also gives some simulation results, which suggest that a better cohesion of the flock can be guaranteed in that way. We also want to point out a third approach to neighborhood definition, based on Voronoi partitions. Its application in biology dates back to~\cite{WDH:71}, and it is well documented in the physicists~\cite{GG-HC-YT:03} and engineers~\cite{FB-JC-SM:09} literatures.

Fourth, each animal only interacts with group mates which are in a suitably defined {\em sensitivity zone}. Restricting the interactions to a sensitivity zone raises the issue of defining its size and shape. On this matter, there is a significant amount of works considering limited visual (or sensing) fields. A limitation in the animal's angle of vision has been incorporated in most models, assuming a blind rear zone~\cite{AH-CW:92,YI-KK:02,IC-JK-RJ-GR-NF:02,CKH-HH:08,RL-YXL-LE:09}, or distinguishing between front and rear sensitivity~\cite{SG-SAL-DIR:96}. In~\cite{HK-CKH:03} authors assume repulsion and alignment regions to be elliptical (taking into account body shapes), and include blind areas.
In the recent paper~\cite{AN-VCB:08}, the authors clearly distinguish between visual field and sensitivity zones, stating that the behavioral rules they use in the model apply only in the front zone.
%In this paper, we adopt this distinction between sensing and sensitivity limitations.
%% qui ho voluto chiarire i termini: sensing=percezione, sentitivity=sensibilità a quello che percepisci
However, up to our knowledge, the anisotropy of sensitivity zones has always been taken just as a given constraint, and not as a potential resource able to shape the group geometry. Here is the main contribution of our paper: showing that a restricted sensitivity angle can be a key element in determining the structure and shape of an animal group. Indeed, we show that changing two parameters, namely two sensitivity angles for attraction and repulsion, we obtain {\em cluster} formations, {\em line} formations and {\em V-like} formations. Note that similar patterns have already been obtained by means of other mathematical models (see for example~\cite{GG-HC-YT:03,SG-SAL-DIR:96,AN-VCB:08}), the novelty here is that our model is able to reproduce all of them, depending on few parameters. We believe that this can offer new biological insights.

Besides investigating the main issue of the role of anisotropy, we also present some results on the effects of limiting the number of considered neighbors, and on the influence of the relative strength of attraction and repulsion on the inter-animal distance. The latter problem relates to the results in~\cite{AM-LE-LB-AS:03}.

The observations drawn from simulations are accompanied by a formal mathematical investigation of the model. One can find that a popular approach to the analysis of agent-based flocking models consists in defining a suitable potential function, called ``virtual'' or ``artificial'' potential, whose gradient gives the dynamics. This variational approach has been followed by both engineers~\cite{HGT-AJ-GJP:07} and biologists~\cite{AM-LE-LB-AS:03,YL-RL-LEK:08} leading to important results. However, it is not possible to apply it to our model. With respect to other models, ours has two main features: {\em state-dependent switching dynamics}, and {\em asymmetry of interactions}.
The former is due to the state-dependent definition of the set of group mates which interact with a given animal. This discontinuous state dependence has been included in previous literature: see for instance the treatment in~\cite{HGT-AJ-GJP:07}, based on non-smooth potentials.
The latter feature is due to the fact that the limited number of interacting neighbors and the restricted sensitivity angles imply that interactions need not to be symmetrical (reciprocal). This fact prevents the application of the virtual potential approach. The paper~\cite{HS-LW-TC:06} is a partial attempt in this direction, because it considers a second-order system in which velocity alignment is achieved by asymmetric interactions. Nevertheless agents' positions are controlled using symmetrical information exchange. As a consequence, its results are not useful to us, as we are interested in spatial configurations. Overall, we conclude that the mathematical analysis of our model requires a novel approach, which has to include discontinuous and asymmetric interactions among animals.

The rest of the paper is organized as follows. The details about the model are given in Section~\ref{sec:model}. Then, extensive simulation results are presented in Section~\ref{sec:results}, whereas Section~\ref{sec:analysis} contains the mathematical analysis. Later, Section~\ref{sec:discussion} discusses the implications of our findings and their biological soundness. We conclude presenting some lines of future research.

%%%%%%%%%%%%%%%%%%%%%%%%%%%%%%%%%%%%%%%%%%%%%%%%%%%%%%%%%%%%%%%%%%%%%%%%%%%%%%%%%%%%%%%%%%%%%%%%%%%%%%%%%%%%%%%%%%%%%%
%%%%%%%%%%%%%%%%%%%%%%%%%%%%%%%%%%%%%%%%%%%%%%%%%%%%%%%%%%%%%%%%%%%%%%%%%%%%%%%%%%%%%%%%%%%%%%%%%%%%%%%%%%%%%%%%%%%%%%
\section{Model definition}\label{sec:model}
The animals in the model are represented by point particles, which have simple continuous-time dynamics.
Given $N\in \natural$, for all $i\in \until{N}$ and $t\in \realnonnegative$, let $x_i(t)\in \real^2$ represent the position of the $i$-th animal, whose evolution is described by the differential equation
\begin{equation}\label{ode_fondamentale}
\dot x_i(t)=v_i(x(t)),
\end{equation}
where $x(t)$ is the vector $(x_1(t),\ldots,x_N(t))$. As in~\cite{AM-LE-LB-AS:03}, we use a coordinate system moving with the group centroid. This means that we are modelling relative movements of the individuals and group structure, rather than its global motion.
The velocity $v_i(x)$ is the sum of two contributions, expressing the effects of attraction and repulsion,
$$v_i(x)=\vc_i(x)-\vr_i(x).$$
In more detail, each of these contributions depends on the relative position of the other animals,
\begin{align*}
\vc_i(x)=&\sum_{j\in \creg{i}{n}}\fc(\|x_j-x_i\|)\frac{x_j-x_i}{\|x_j-x_i\|}\\
\vr_i(x)=&\sum_{j\in \rreg{i}{n}}\fr(\|x_j-x_i\|)\frac{x_j-x_i}{\|x_j-x_i\|}.
\end{align*}
The following definitions have been used.
\begin{itemize}
\item The function $\map{\fc}{\realnonnegative}{\realnonnegative}$  (resp., $\map{\fr}{\realnonnegative}{\realnonnegative}$) describes how each animal is attracted (resp., repelled) by a neighbor at a given distance, assuming $\|\cdot\|$ denotes the Euclidean norm in $\real^2$.
\item The {\em attraction neighborhood} $\creg{i}{n}$ (resp., the {\em repulsion neighborhood} $\rreg{i}{n}$) is the set of the $n$ animals closest to the $i$-th one, which are inside the attraction (resp., repulsion) sensitivity zone.
\end{itemize}

The above model is very general, and we need to specialize it by choosing the interaction functions and the shape of the sensitivity zones. We make the following assumptions.
\begin{enumerate}
%\item[A1.] The functions $\fc$ and $\fr$ are assumed to be $\displaystyle \fc(s)= \Fc s$ and $\displaystyle \fr(s)=\frac{\Fr}{s},$ where $\Fc$ and $\Fr$ are two positive constants.
\item[A1.] The functions $\fc$ and $\fr$ are assumed to be 
$$ 
\fc(\|x_j-x_i\|)= \Fc \|x_j-x_i\|\,, \qquad \fr(\|x_j-x_i\|)=\frac{\Fr}{\|x_j-x_i\|},
$$ 
where $\Fc$ and $\Fr$ are two positive constants.
\item[A2.] The sensitivity zones are depicted in Figure~\ref{fig:regions} and illustrated as follows. Let the center point be the animal's position, and let the horizontal axis (arrow-headed) represent the direction of motion. Attraction is active in a frontal cone whose width is given by the angle $\alphac\in(0,360\degs]$ (dashed line). Repulsion is active both inside a disk of radius $\Rsr>0$ (\emph{short-range} repulsion) and in a frontal cone of width $\alphar\in(0\degs,360\degs]$ (solid line). We stress that the sensitivity zones do not necessarily coincide with the visual field of the animal. They rather represent the zones which attraction and repulsion are focused on.
\item[A3.] The speed of each animal $\|v_i\|$ is bounded from above by a constant $\vmax.$
\end{enumerate}

\begin{figure}[!ht]\centering
\includegraphics[width=.6\columnwidth]{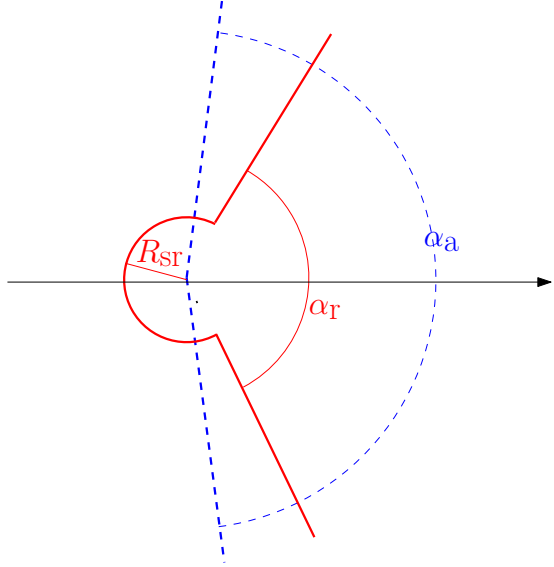}
  \caption{The shape of the sensitivity zones.} \label{fig:regions}
\end{figure}
The above assumptions result in the system
\begin{equation}\label{eq:system-simulated}
\dot x_i(t)=\Fc\sum_{j\in \creg{i}{n}} (x_j-x_i)  -  \Fr \sum_{j\in \rreg{i}{n}} \frac{(x_j-x_i)}{\|x_j-x_i\|^2} .
\end{equation}
Some remarks are in order.
\begin{enumerate}
%\item[R.] Although the model has dynamics in $\real^2$, extension to higher dimensions is straightforward.
\item[R1.] The definition of the interaction neighborhoods $\creg{i}{n}$ and $\rreg{i}{n}$ allows to have a priori bound on the number of effective neighbors, and therefore on the sensing and ``computational'' effort which is required for each animal. This fact, which copes with animals' intrinsic limitations, has been experimentally observed in biology, for fish~\cite{IA:80} and birds~\cite{AC-IG:08a}. The latter paper calls this neighborhood definition {\em topological}, as opposed to {\em metric} definitions, based on distance only.
\item[R2.] Our assumption of unbounded sensitivity regions does not intend to imply that animals sensing capabilities extend on an unlimited range, but rather that group dynamics happen in a relatively small area.
\item[R3.] If $\alphac=\alphar=360\degs$, the parameter $\Rsr$ has no effect, and the interaction is completely isotropic as in the simulations presented in~\cite{AC-IG:08a,DC-SAL:99}. Note that, even if there is no preference for any specific direction, the limitation of the number of considered neighbors makes the interactions not reciprocal, i.e. the fact that the $i$-th animal interacts with the $j$-th does not imply that the $j$-th interacts with the $i$-th.
\item[R4.] By specializing the functions $\fc$ and $\fr$, one can obtain various interaction models. %For instance, it is natural to assume that $\fr$ be decreasing and $\lim_{s\to \infty}{\fr(s)}=0$.
Indeed many proposals can be found in the literature, as reviewed in~\cite{KW-JL:91} and~\cite{AM-LE-LB-AS:03}. However, the shape of these functions is not the main point in our paper, and thus we have decided to focus on a simple choice, in order to highlight the innovative part of our approach, i.e. the angle-dependent interactions.
Similar considerations are valid also for some features introduced in other models, such as a \textit{neutral zone} around animals~\cite{JHT-SAL-DIR:04,DC-SAL:99} or a hierarchical decision tree which allows the repulsion force to have the priority over the attraction force~\cite{SG-SAL-DIR:96,DC-SAL:99}.
\end{enumerate}

\section{Simulations results}\label{sec:results}
In this section we make use of the agent-based model~\eqref{eq:system-simulated} in order to show the effects of the anisotropic interactions on the shape of the group. Namely, depending on the angles $\alphac$, $\alphar$, and the ratio between repulsive and cohesive forces, we shall obtain either {\em clusters}, or {\em lines} or {\em V-like} formations. These patterns are described in the sequel.

To perform the simulations, it is instrumental the introduction of the constant $\xi=\sqrt{\frac{\Fr}{\Fc}}$, which allows to rewrite~\eqref{eq:system-simulated} as
\begin{equation}\label{eq:system-adim}
\dot x_i(t)=\sum_{j\in \creg{i}{n}} (x_j-x_i)  -  \xi^2 \sum_{j\in \rreg{i}{n}} \frac{(x_j-x_i)}{\|x_j-x_i\|^2}.
\end{equation}
In equation~\eqref{eq:system-adim}, the unit of length is chosen to be the body length (\bl) of the animal, and the time unit (\tu) is the inverse of $\Fc$.
Simulations are obtained solving the system of equations~\eqref{eq:system-adim} via an explicit forward adaptive Euler scheme. Each run starts from a randomly generated initial configuration (contained in square of edge $\Lci$), and ends when the system reaches a steady state. To take into account the uncertainties in sensing and motion of the animals, we include small additive random disturbances on the direction of the velocity, uniformly distributed in $[-\alpha_{\textrm{noise}},\alpha_{\textrm{noise}}]$. All the steady-state configurations described in the sequel are robust to such noise. A summary of the parameters and their values is given in Table~\ref{table:parameters}.
In the sequel we discuss the role of $N$, $n$, $\xi$, $\alphac$ and $\alphar$, which are most significant to us, whereas $\Rsr$, $\vmax$, $\alpha_{\textrm{noise}}$ and $L$ are kept fixed.

\begin{table}\caption{Model parameters.}\label{table:parameters}
\begin{tabular}{|c|c|c|c|}
  \hline
  Name & Symbol & Unit & Values explored \\  \hline  \hline
  Forces ratio      & $\xi$ & \bl                               & 0.1\textendash{20} \\  \hline
  Attraction angle    & $\alphac$ & degrees $(\degs)$                & 0\textendash{}360 \\  \hline
  Repulsion angle   & $\alphar$ & degrees $(\degs)$            & 0\textendash{}360 \\  \hline
  Number of animals & $N$ & adimensional                    &  2\textendash{}200 \\  \hline
  Number of considered neighbors & $n$ & adimensional                  & 1\textendash{}$(N-1)$ \\  \hline
  Short-range repulsion radius & $\Rsr$ & \bl             & 1 \\  \hline
  Maximum speed & $\vmax$    & \bl/\tu                                 &  2\textendash{}30 \\ \hline
  %Repulsion cut-off  & $\eps$ & \bl                     & 0\textendash{}20 (lines only) \\  \hline
  Size of initial domain & $\Lci$ & \bl                     & 15 \\  \hline
  Noise magnitude& $\alpha_{\textrm{noise}}$ & degrees $(\degs)$                        & 0\textendash{}10 \\ \hline
  %\hline
\end{tabular}
\end{table}

\subsection{Clusters}
In this paragraph, we describe simulation results when interactions are assumed to be isotropic. These results are not dissimilar from others in the literature (e.g.,~\cite{AM-LE-LB-AS:03,YL-RL-LEK:08}): we include them for two reasons. First, for comparison with the less usual patterns described in the following paragraphs. Second, because they allow some interesting remarks about the role of the model's parameters $\xi$ and~$n$.

Thus, let us assume that $\alphac=\alphar=360\degs$. As a consequence, the outcome of the simulations is a cohesive and well spaced cluster. For a better understanding, we make use of an indicator which is largely used in the literature (see for instance~\cite{AH-CW:92,HK-CKH:03}): the {\em mean distance to the nearest neighbor} NND, defined as
$$
\textrm{NND}=\frac{1}{N}\sum_{i=1}^N \min_{j\neq i} \|x_i-x_j\|.
$$
We investigate the dependence of NND on the parameter $\xi$: simulation results are shown in Figure~\ref{fig:2crystal}.
\begin{figure}[!ht]\centering
\includegraphics[width=.49\columnwidth]{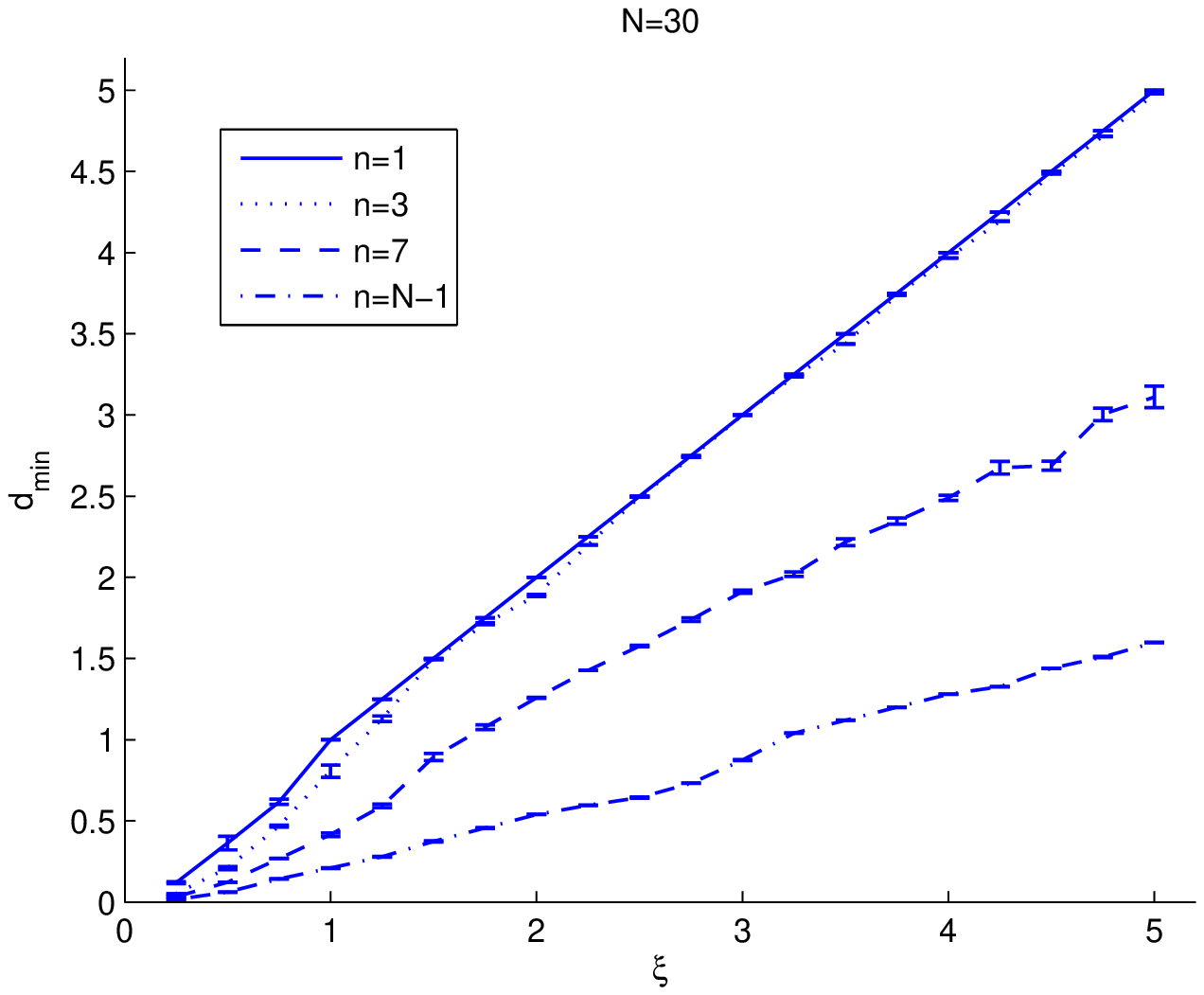}
\includegraphics[width=.49\columnwidth]{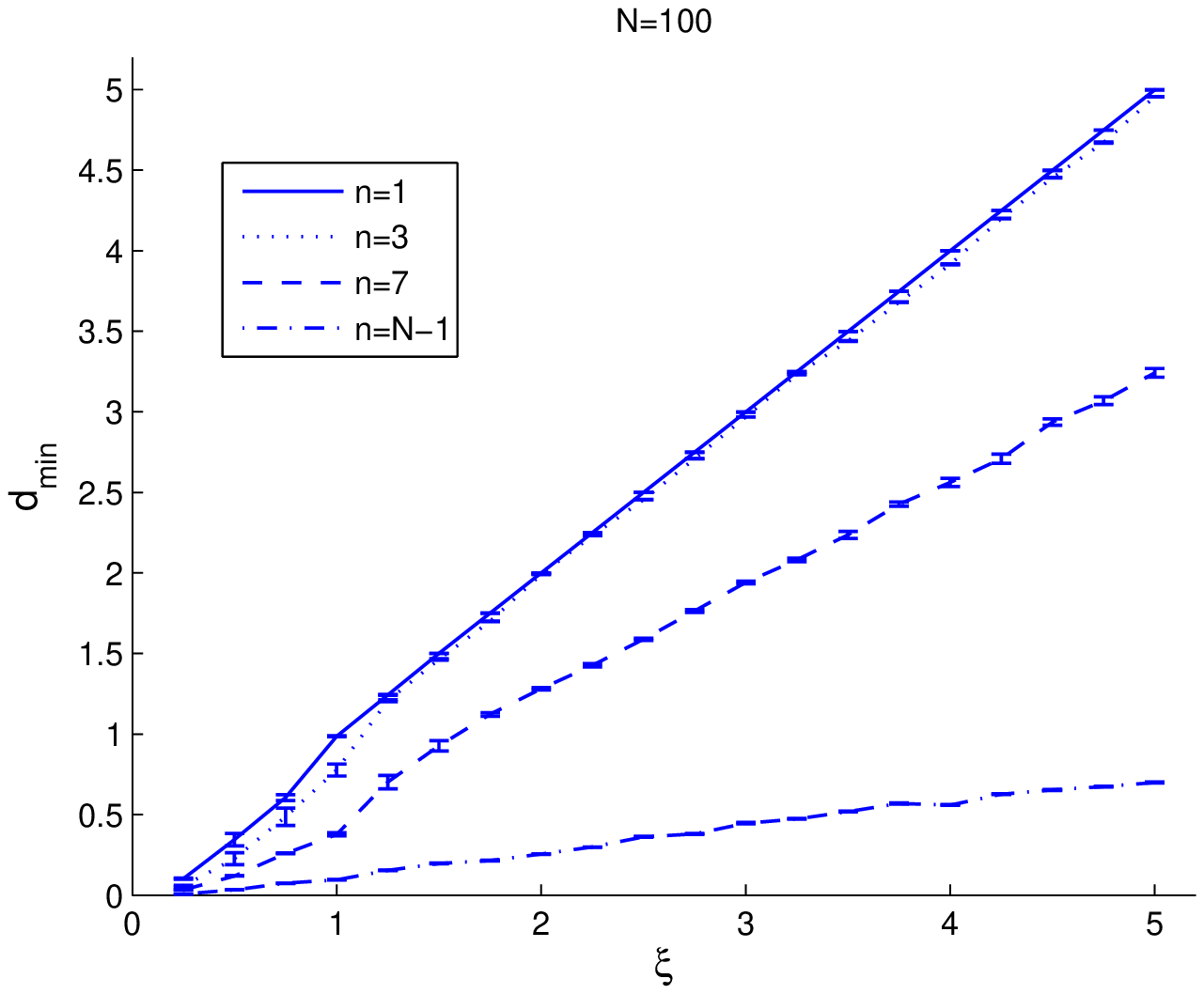}
  \caption{NND as a function of the ratio $\xi$, for different values of $n$. Error bars denote variance across individuals. Plots assume $N=30$ (left) and $N=100$ (right).} \label{fig:2crystal}
\end{figure}
We observe that NND is an increasing function of $\xi$, in particular it increases roughly linearly in $\xi$. This linear dependence is observed for any choice of $n$. Two other features are noticeable: first, if $n=1$, animals asymptotically converge to a {\em comfortable distance} which is equal to $\xi$. Second, for any fixed $\xi$, NND decreases as $n$ increases. Moreover, all these remarks do not depend on $N$.

These results can be compared with those in~\cite{AM-LE-LB-AS:03}: in that paper, the authors assume that all animals in the group interact among each other, and they conclude that, the larger the group, the closer packed it is. Our simulations, instead, suggest that the significant parameter is $n$, the number of neighbors which is taken into account, rather than $N$, the global number of animals.  However, from the biological point of view, the number of neighbors $n$ is not truly a free parameter: $n$ can not be too large because of the limited sensing and analysis capabilities of the animals, and can not be too small either. For instance, a low value of $n$ sounds unsafe from the point of view of collision avoidance. %Indeed, experimental evidence suggests $n$ be approximately between 2-3 and 7, depending on species~\cite{IA:80,AC-IG:08e}.

Figure~\ref{fig:2crystal} also shows that the variance of the distance to the nearest neighbor among the animals is quite small. This means that groups are internally uniform, in terms of spacing. This uniformity is also apparent if one looks at the steady-state configurations: three examples are shown in Figure~\ref{fig:cluster-views}.
\begin{figure}[!ht]\centering
\subfloat[Crystal-like cluster: $n=1$]{\label{fig:cluster-1} \includegraphics[width=.33\columnwidth]{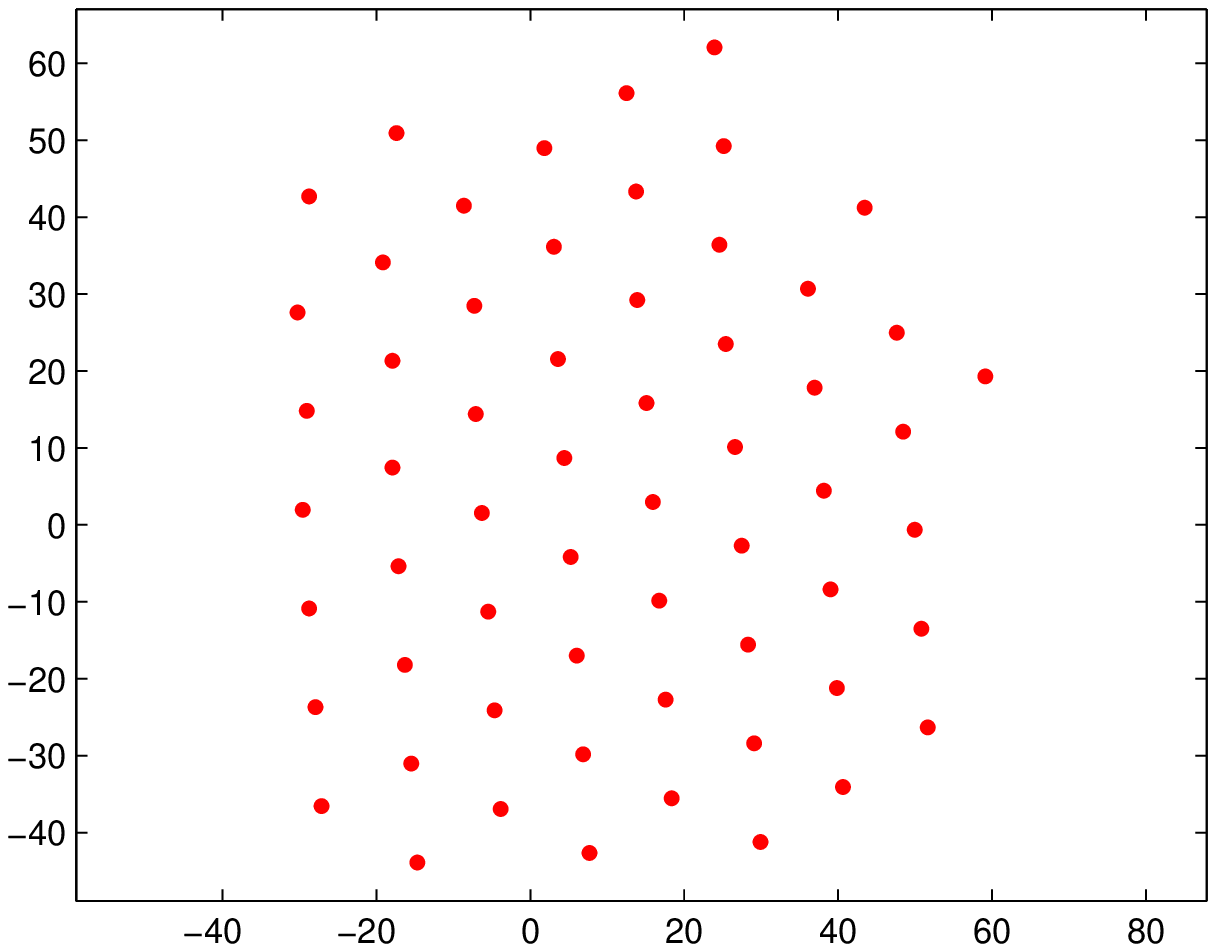}}
\subfloat[Disordered cluster: $n=7$]{\label{fig:cluster-7} \includegraphics[width=.33\columnwidth]{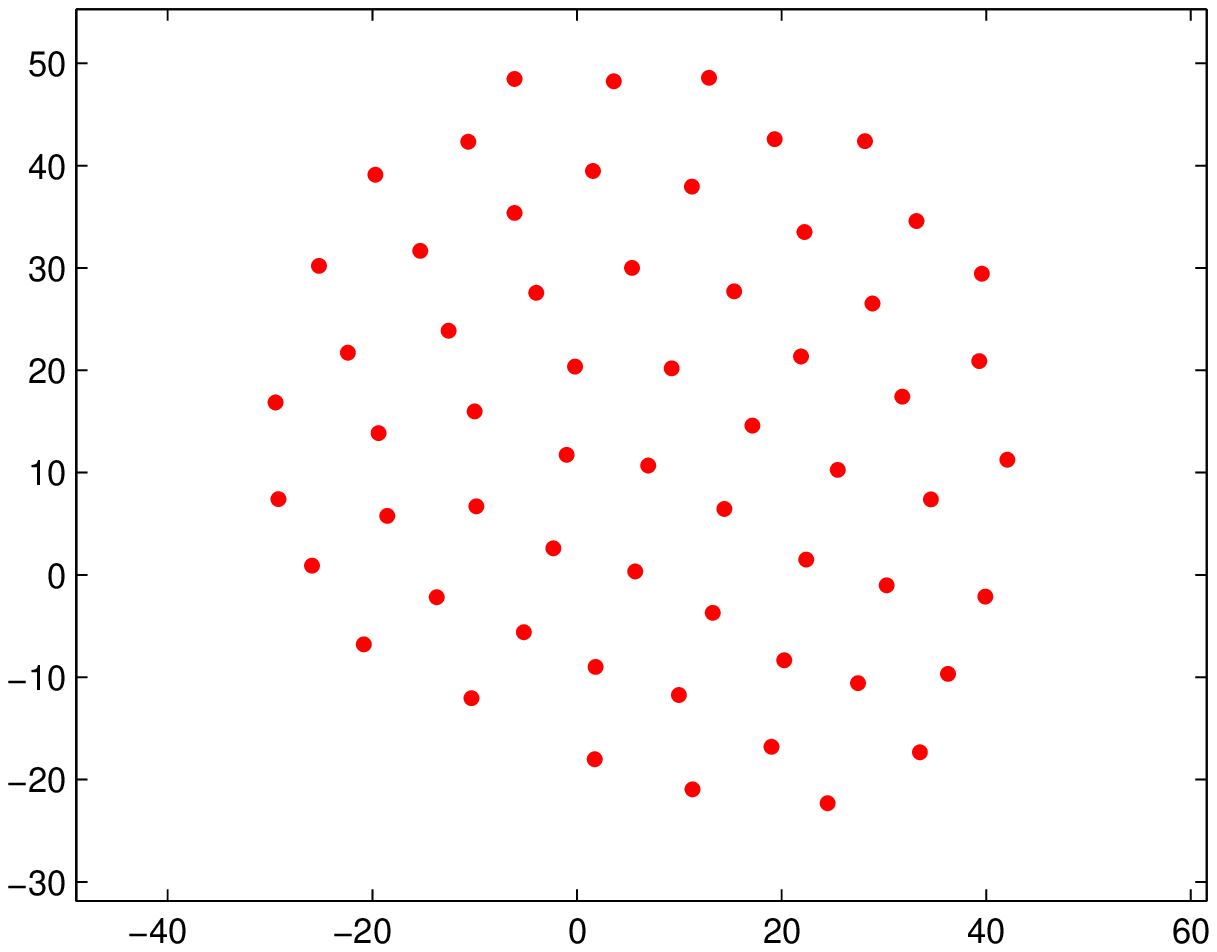}}
\subfloat[Round cluster: $n=N-1$]{\label{fig:cluster-N} \includegraphics[width=.33\columnwidth]{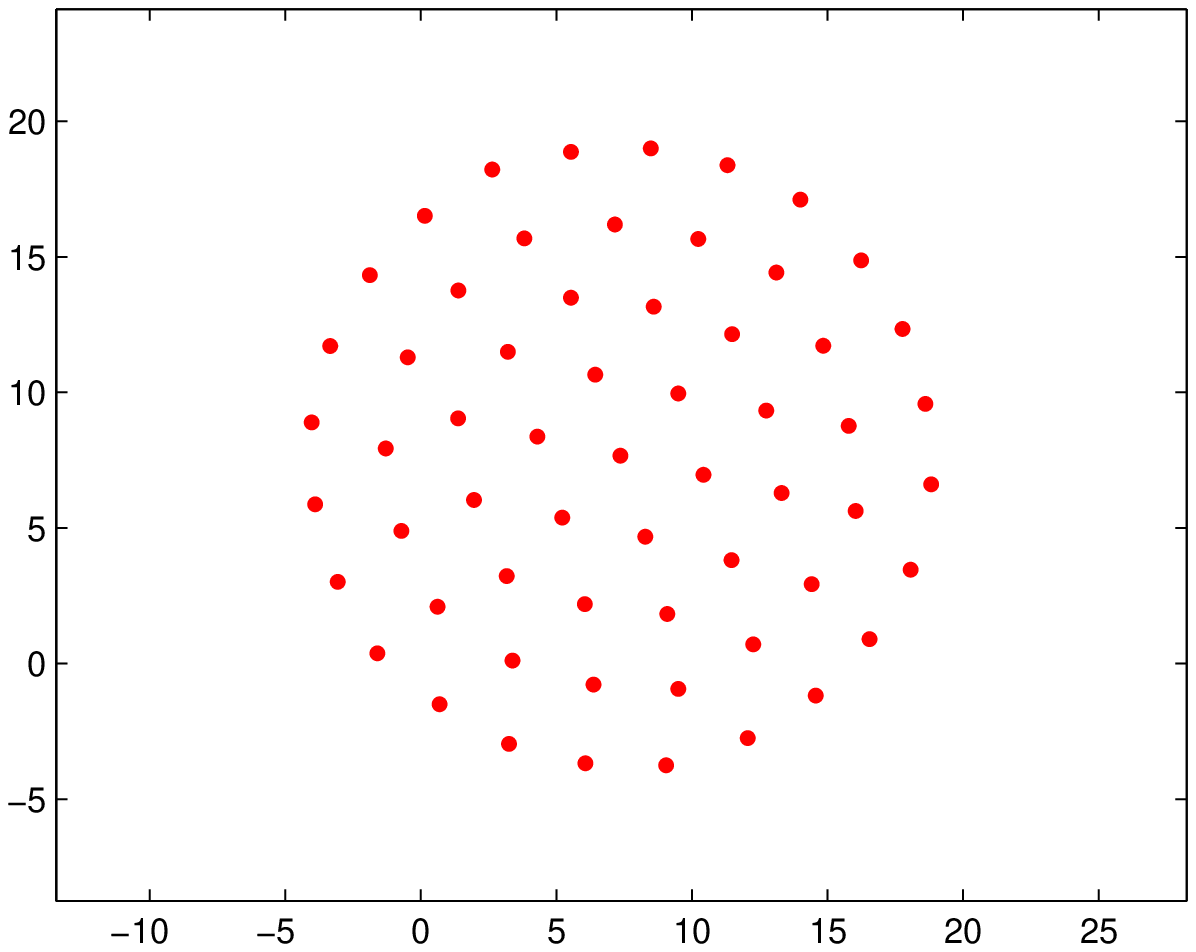}}
  \caption{Examples of internal structures of clusters obtained with $\xi=10$ and $N=60$. The groups are moving horizontally from left to right.} \label{fig:cluster-views}
\end{figure}
Nevertheless, as $n$ varies, significant modifications are apparent in terms of relative positions. If $n=1$, the animals are deployed to form an hexagonal lattice, reminiscent of a crystal. For intermediate $n$'s the internal structure is rather disordered, and if $n=N-1$ it is made of concentric circles. Note that crystal-like patterns have also been found in~\cite{GG-HC-YT:03} and in~\cite{YL-RL-LEK:08}, where they are compared with the less regular structures obtained in~\cite{YC:07}. Such patterns might be of interest for engineering application to environmental deployment of robots~\cite{FB-JC-SM:09} or sensors~\cite{PF-PM-BP:09a}.

\subsection{Lines}\label{subsec:lines}
In this paragraph we show how restricting the sensitivity field (i.e. reducing $\alphar$ and $\alphac$) induces the formation of an elongated group. The key element for the formation of these patterns is a restricted frontal sensitivity field. Let us denote by $e$ the oriented {\em elongation} of the group (see for instance~\cite{SG-SAL-DIR:96,HK-CKH:03}), defined as the ratio of the vertical to the horizontal side of the smallest rectangle containing the group, oriented parallel to the direction of the movement.
Here we study how $e$ depends on the angles ($\alphar$, $\alphac$), and on $n$. Results are summarized in Figure~\ref{fig:elongation}, which shows the average and the extreme values of $e$ over 100 runs as a function of the angles ($\alphar$, $\alphac$), for $n=1,7,N-1$. It is clear that reducing angles from $(360\degs,360\degs)$ to $(40\degs,180\degs)$ affects the elongation of the group.
\begin{figure}[htb]\centering
\includegraphics[width=.49\columnwidth]{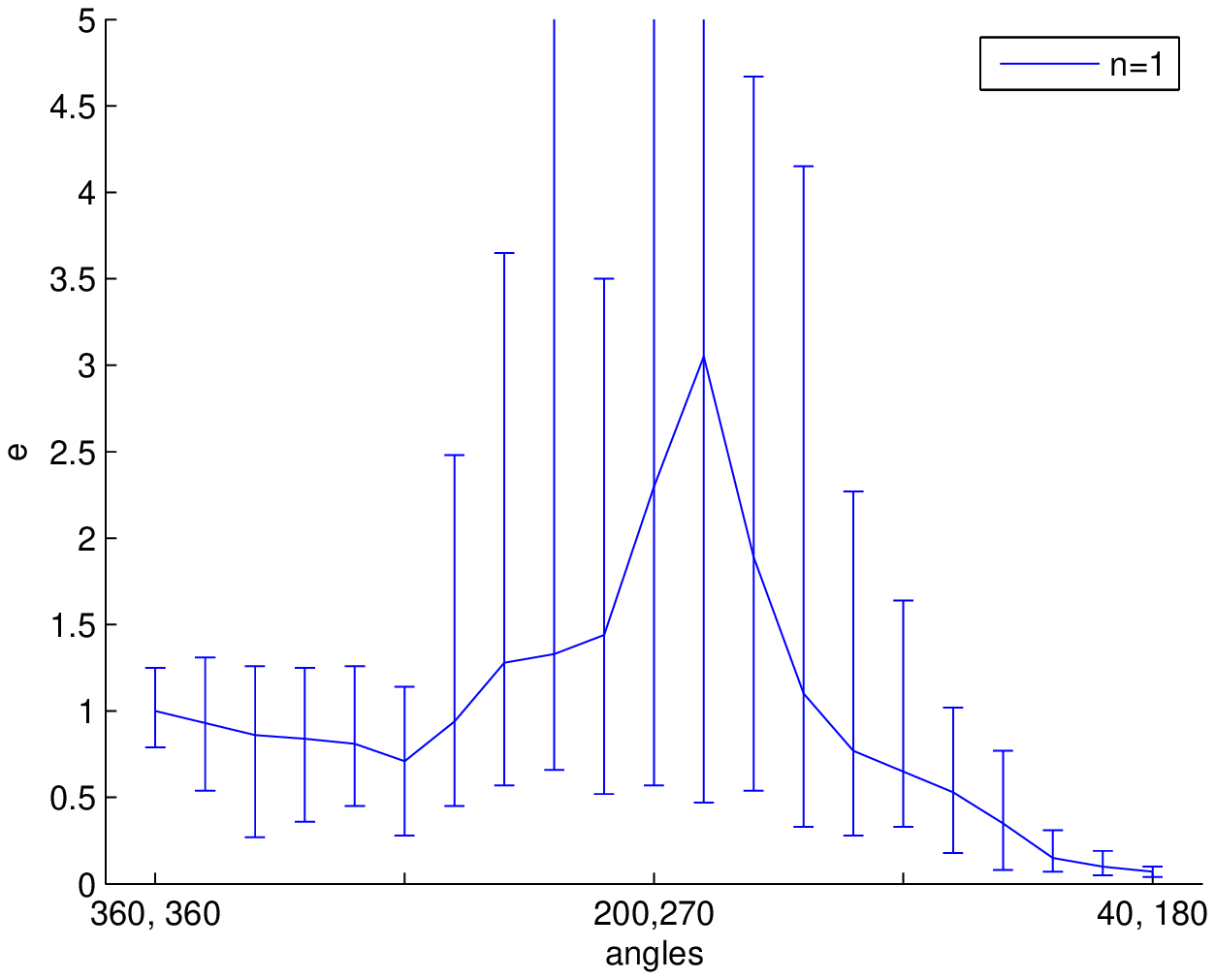}
\includegraphics[width=.49\columnwidth]{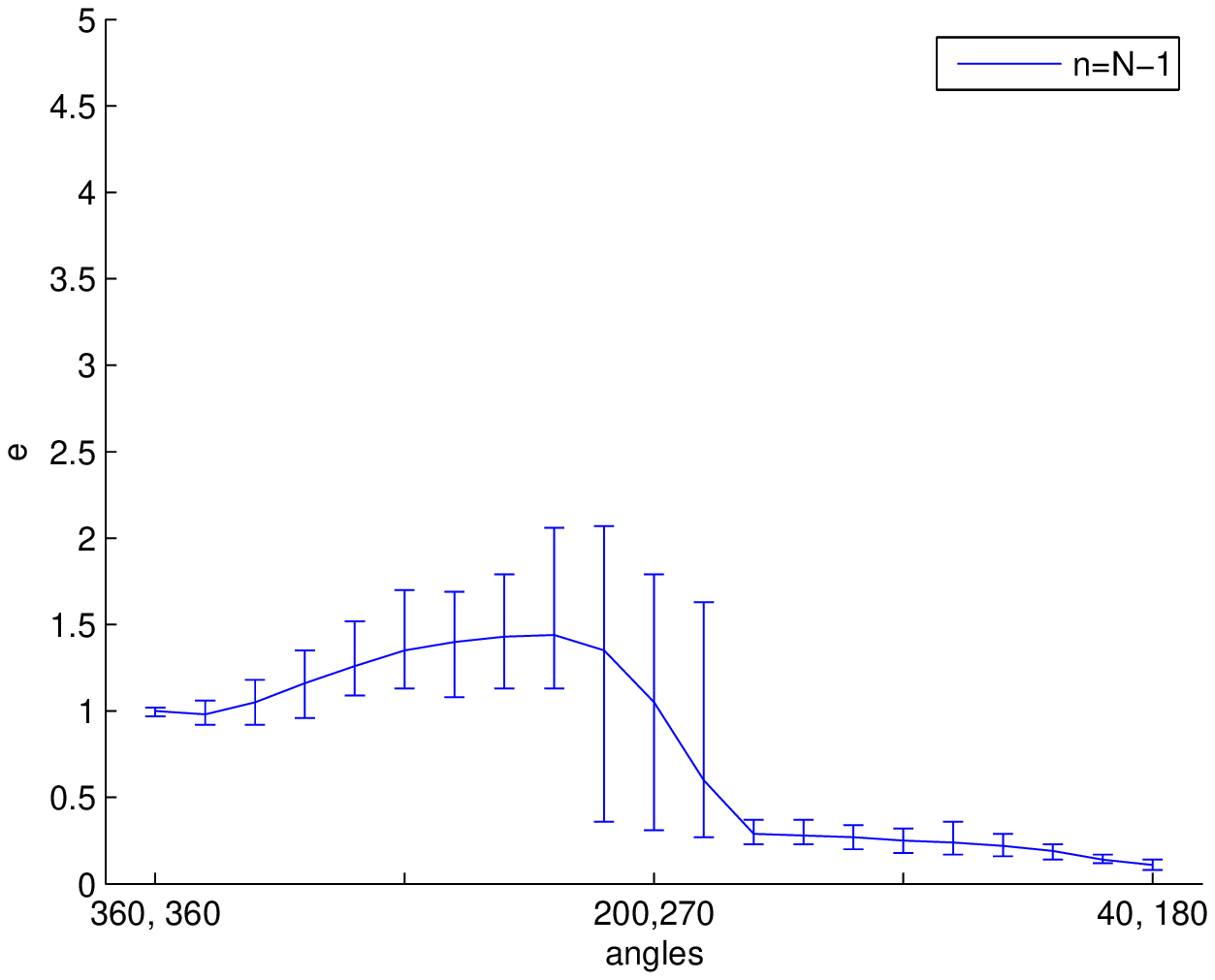}\\
\includegraphics[width=.7\columnwidth]{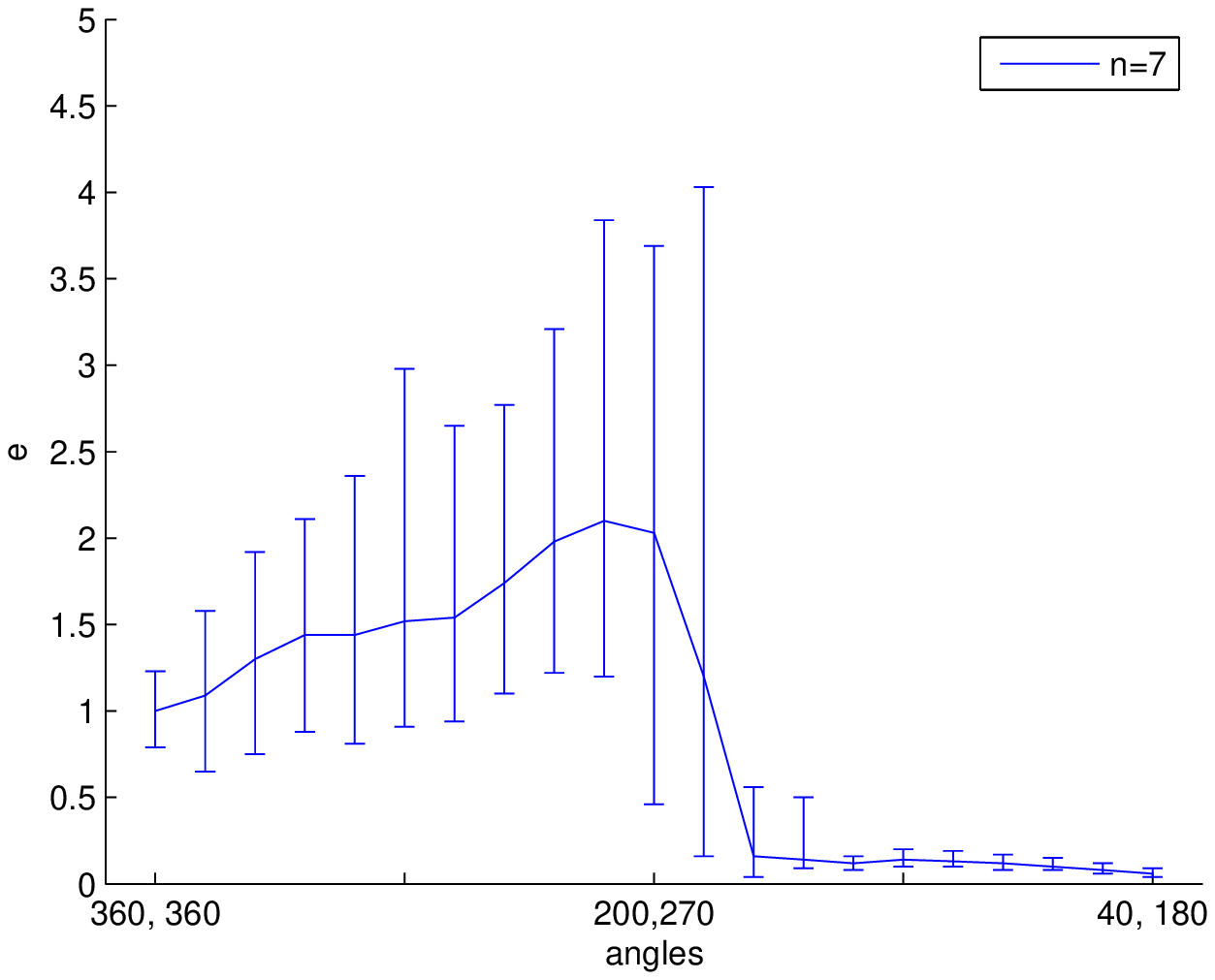}
\caption{Mean elongation of the group as a function of the sensitivity angles, for $N=30$ and different values of $n$. The angles are $(\alphar$, $\alphac)=(360\degs-k 16\degs,360\degs-k 9 \degs)$, $k=0,\dots,20$. Data come from 100 runs. Error bars are the ranges of the outcomes.}
\label{fig:elongation}
\end{figure}
\begin{figure}[htb]\centering
\includegraphics[width=.7\columnwidth]{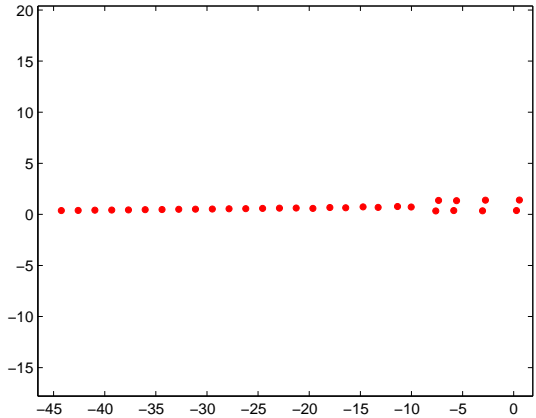}
  \caption{A line formation, obtained with $\alphar=40\degs$, $\alphac=180\degs$, $n=7$, $\xi=10$, $N=30$. The group is moving horizontally from left to right. See text for the explanation of the irregularities in the head of the line.}
  \label{fig:line}
\end{figure}
Wide angles (roughly, $360\degs>\alphar>200\degs$ and $360\degs>\alphac>270\degs$) induce an average elongation greater than 1, i.e. the group stretches along the transverse direction. Moreover, the range of the outcomes is large, meaning that different initial conditions affect significantly the evolution of the system. In the majority of runs the system does not reach an equilibrium in a reasonable time, and the simulation is stopped after a maximum number of iterations.
Conversely, small angles ($200\degs>\alphar>40\degs$ and $270\degs>\alphac>180\degs$) lead to small values of $e$, with small differences among the runs. The steady-state configurations are strongly elongated in the direction of motion: in the limit case, we obtain a line as in Figure~\ref{fig:line}. One can also observe a dependence on $n$: larger values of $n$ show a sharper transition to lines than with $n=1$.
The obtained lines are not in general stable: this fact relates to the phenomenon of {\em string instability} described in the engineering literature~\cite{PS-AP-JKH:04}: in a line of vehicles which are tracking their forerunners positions, perturbations propagate down the line in cascade, leading to instabilities. Control-theoretic results presented in~\cite{DS-JKH:96} state that weakening the interaction forces should reduce these negative effects. Consistently, we have observed that damping down to zero small repulsion forces improves the stability of the lines.

Finally, we note that in Figure~\ref{fig:line} some ``border effects'' are visible in the head of the line (right side): since in this case we chose $n=7$, the animals in the front can not interact with a sufficient number of group mates and then they do not form a single-file line.

\subsection{Vees}
In this paragraph we show that a restricted frontal {\em repulsion} range ($\alphar<180\degs$) induces the formation of V-like patterns. V-like formations have been recently obtained in the literature~\cite{AN-VCB:08}, using an {\it ad hoc} model motivated by aerodynamics considerations. In our model, instead, V-like formations arise as one of the anisotropy's effects.
Following~\cite{FHH:74}, we adopt a broad definition of V-like formations, which includes asymmetric formations (J-like, and echelons) as well. To understand the role of anisotropy in the emergence of such patterns, we study how the configurations depend on the repulsion angle $\alphar$, while we keep fixed $\alphac=360\degs$.\footnote{To obtain sharper evidence from simulations, it is useful to strengthen repulsion with respect to attraction, taking a value of $\xi$ larger than for lines: in the following simulations we set $\xi=13$, and we also set $n=7$ and $\vmax=10$.
%However, its exact value is not critical for our results.
%We also note that, while for clusters and lines the value of $\vmax$ has no influence, such a choice can affect the convergence to V-like formations. However, neither the exact value of $\vmax$ is critical.
}

In Figure~\ref{fig:roses} we plot the distribution of the angles between the nearest neighbor and the direction of motion, for four values of $\alphar$.
A few remarks are in order. If $\alphar=360\degs$, animals do not show any angle preference. If $\alphar=270\degs$, animals show a preference of the front/back positions versus side positions. Finally, if $\alphar<180\degs$, animals clearly prefer to keep a specific angle, namely $\frac{\alphar}{2}$, with respect to their nearest neighbor.

\begin{figure}[h!]\centering
\includegraphics[width=.49\columnwidth]{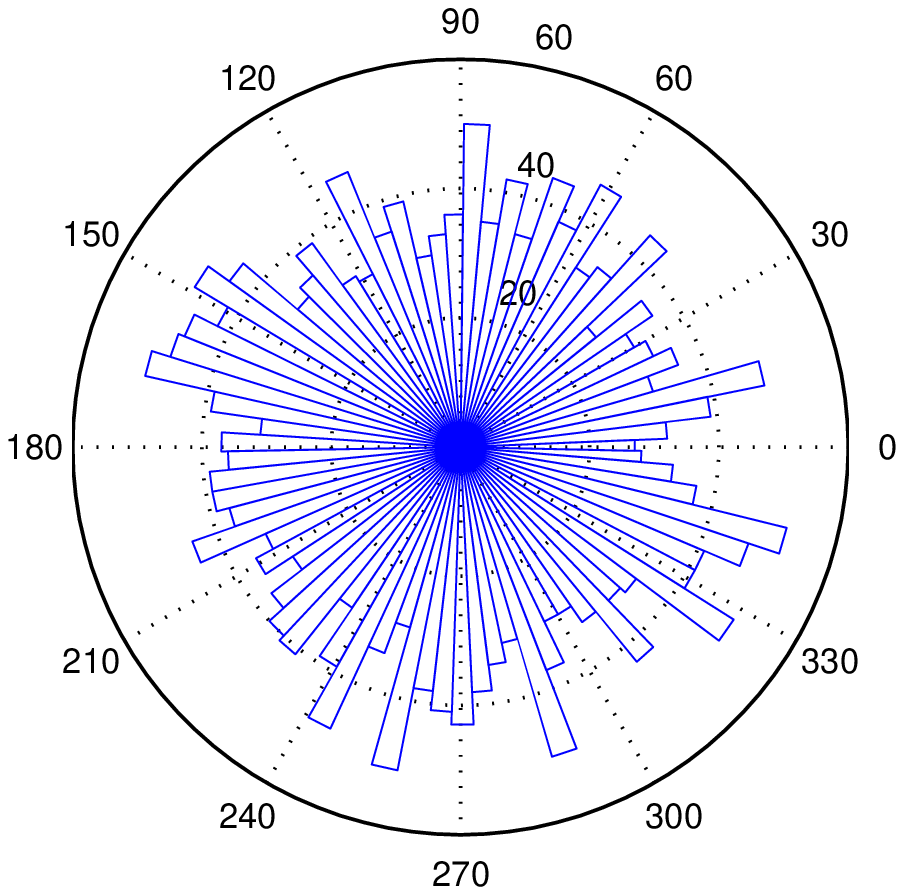}
\includegraphics[width=.49\columnwidth]{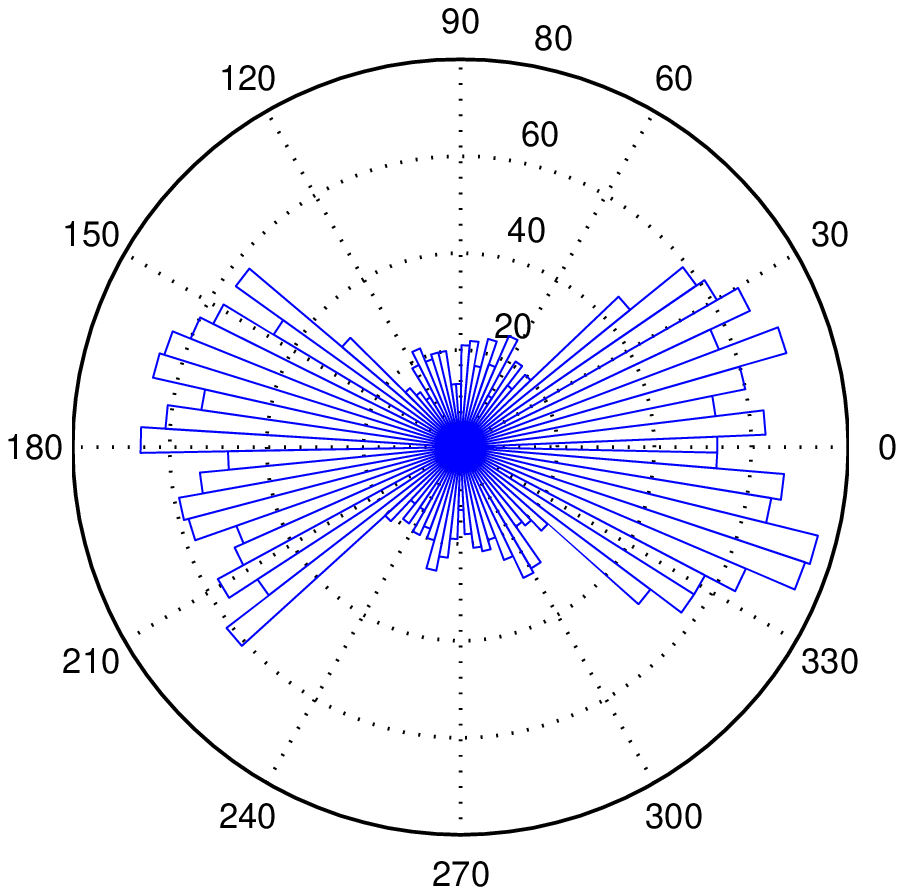}
\includegraphics[width=.49\columnwidth]{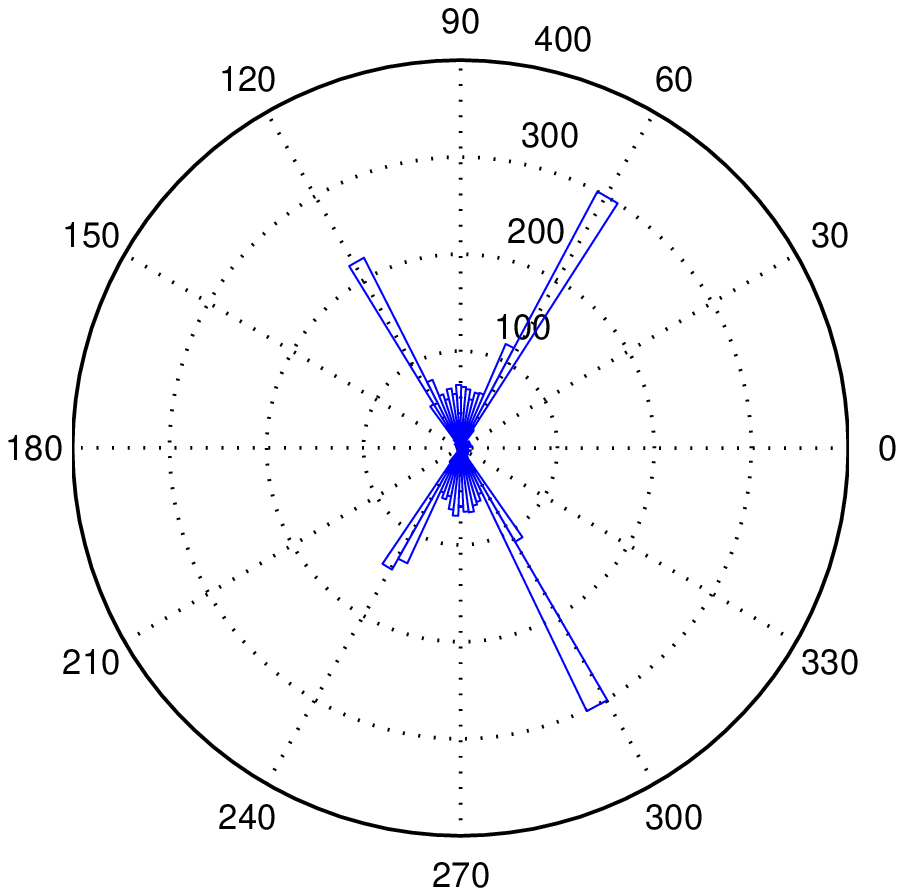}
\includegraphics[width=.49\columnwidth]{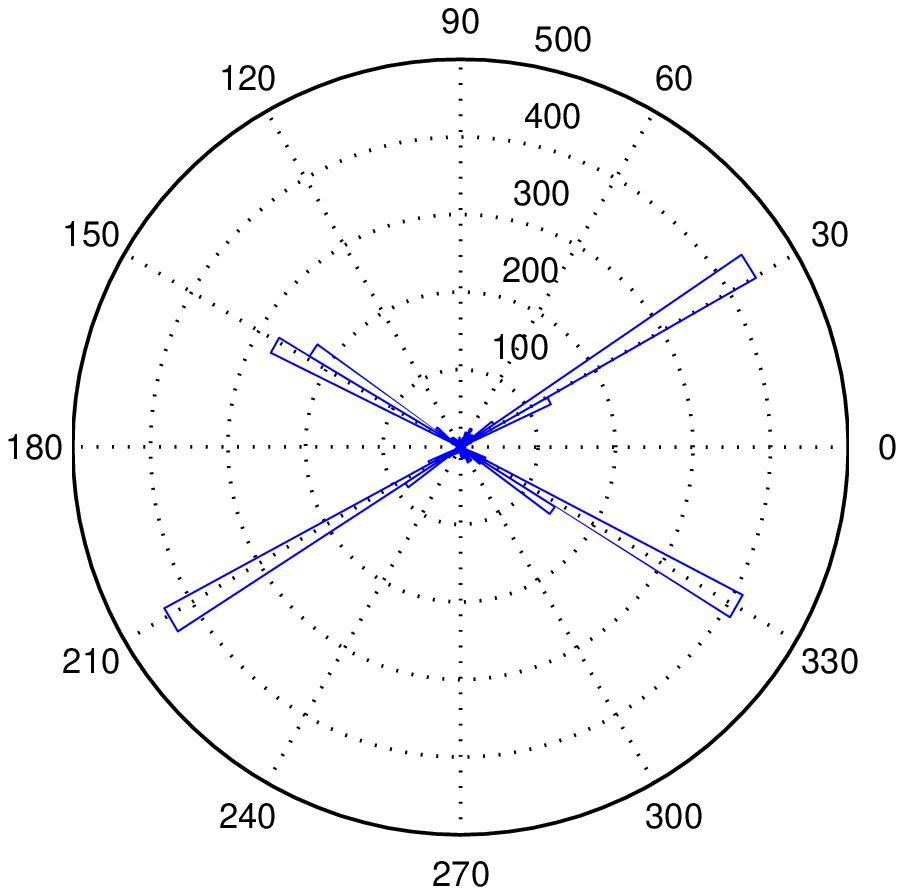}
  \caption{Distribution of the nearest\textendash{}neighbor angle for $\alphar=360\degs$, $270\degs$, $120\degs$, $60\degs$. Data from 100 runs.} \label{fig:roses}
\end{figure}

Second, to obtain more quantitative results about the transition from clusters to Vees, we introduce an Alignment Index ($\ai$), defined as follows. $\ai(\theta)$ is the percentage of individuals whose nearest neighbor is positioned (up to a small tolerance $\angleTol$) at a given angle $\theta$ with respect to them. The dependence on $\alphar$ of this novel index is shown in Figure~\ref{fig:angles}. The figure plots both $\ai(\alphar/2)$ and $\ai(30\degs)$, computed as the average over 100 runs, with $\angleTol=3\degs$. If $\alphar>180\degs$ the alignment index is as low as in a random configuration; if instead $\alphar<180\degs$, the high index confirms the preference for an $\frac{\alphar}{2}$-alignment.
\begin{figure}[h!]\centering
\includegraphics[width=.7\columnwidth]{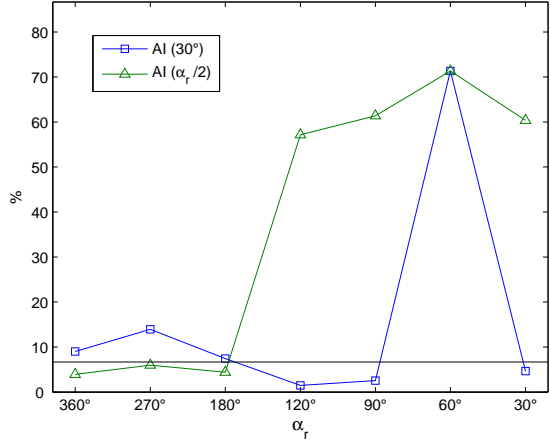}
  \caption{ $\ai(30\degs)$ and $\ai(\alphar/2)$ as functions of $\alphar$. Average of 100 runs. The reported reference value $6\%$ is the expected value of $\ai$ from a random uniform distribution.}
  \label{fig:angles}
\end{figure}
Qualitative analysis of the obtained configurations confirm these results: one observes that for a wide range of $\alphar$, a scenario sets up, in which the animals form (several) V-like formations. Examples are given in Figure~\ref{fig:Vee}.
Finally, it can be noted from Figure~\ref{fig:shots} that the number of considered neighbors $n$ affects the ability to form V-like configurations: if $n$ is too small ($n=1$) or too large ($n=N-1$), the interesting patterns do not show up.
\begin{figure}[h!]\centering
\includegraphics[width=.45\columnwidth,bb=80 250 550 590]{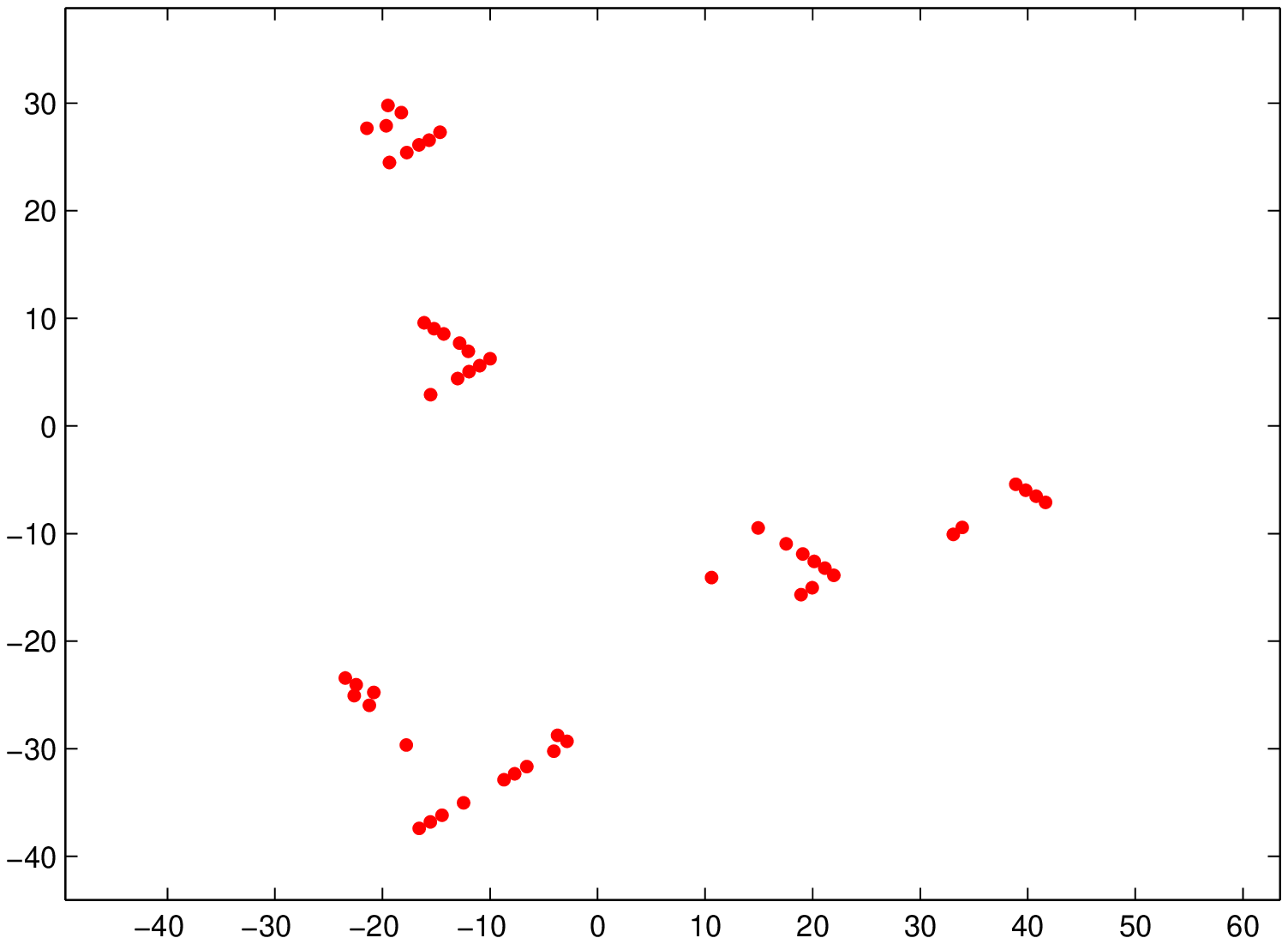}
\includegraphics[width=.40\columnwidth,bb=150 290 430 520]{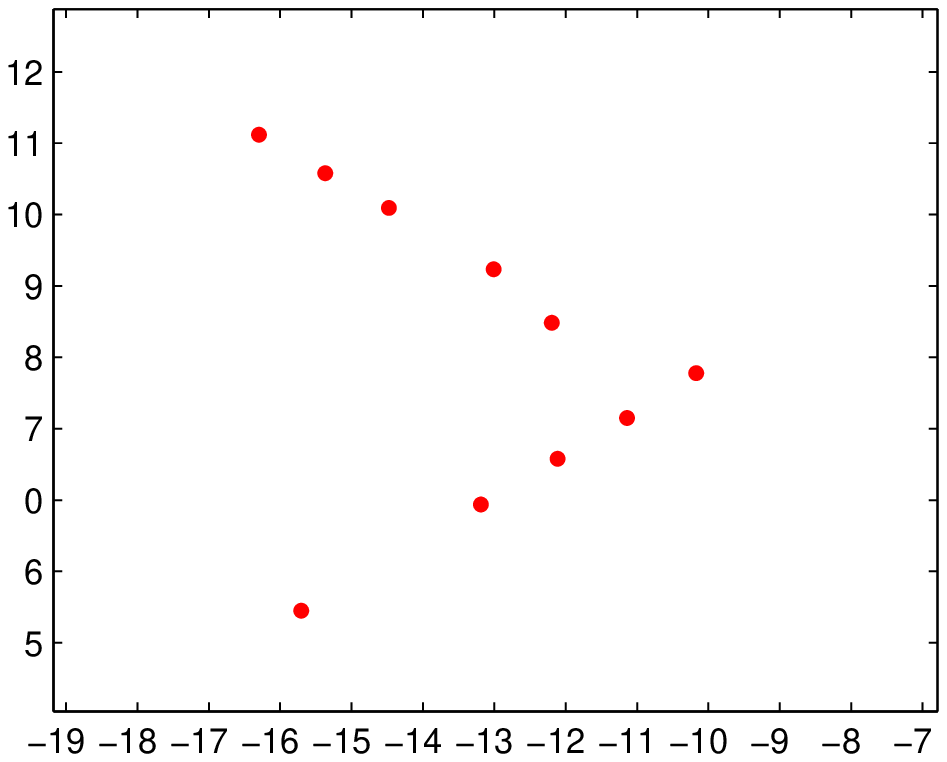}
  \caption{V-like formations obtained with $\alphar=60\degs,$ $N=30$. The plot on the right is a close-up on one of the Vees. The group is moving horizontally from left to right.}
  \label{fig:Vee}
\end{figure}

\begin{figure}[h!]\centering
\includegraphics[width=.49\columnwidth]{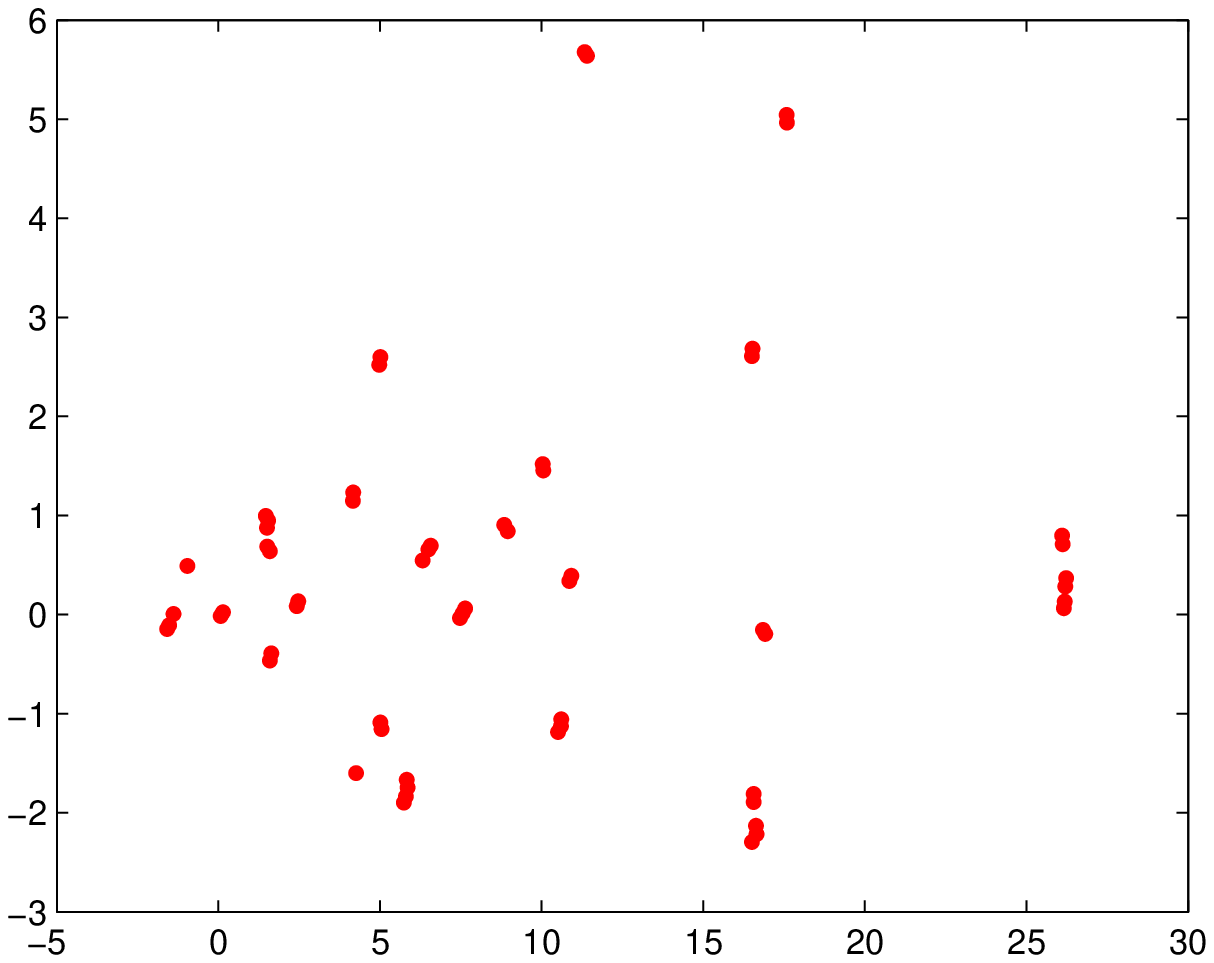}
\includegraphics[width=.49\columnwidth]{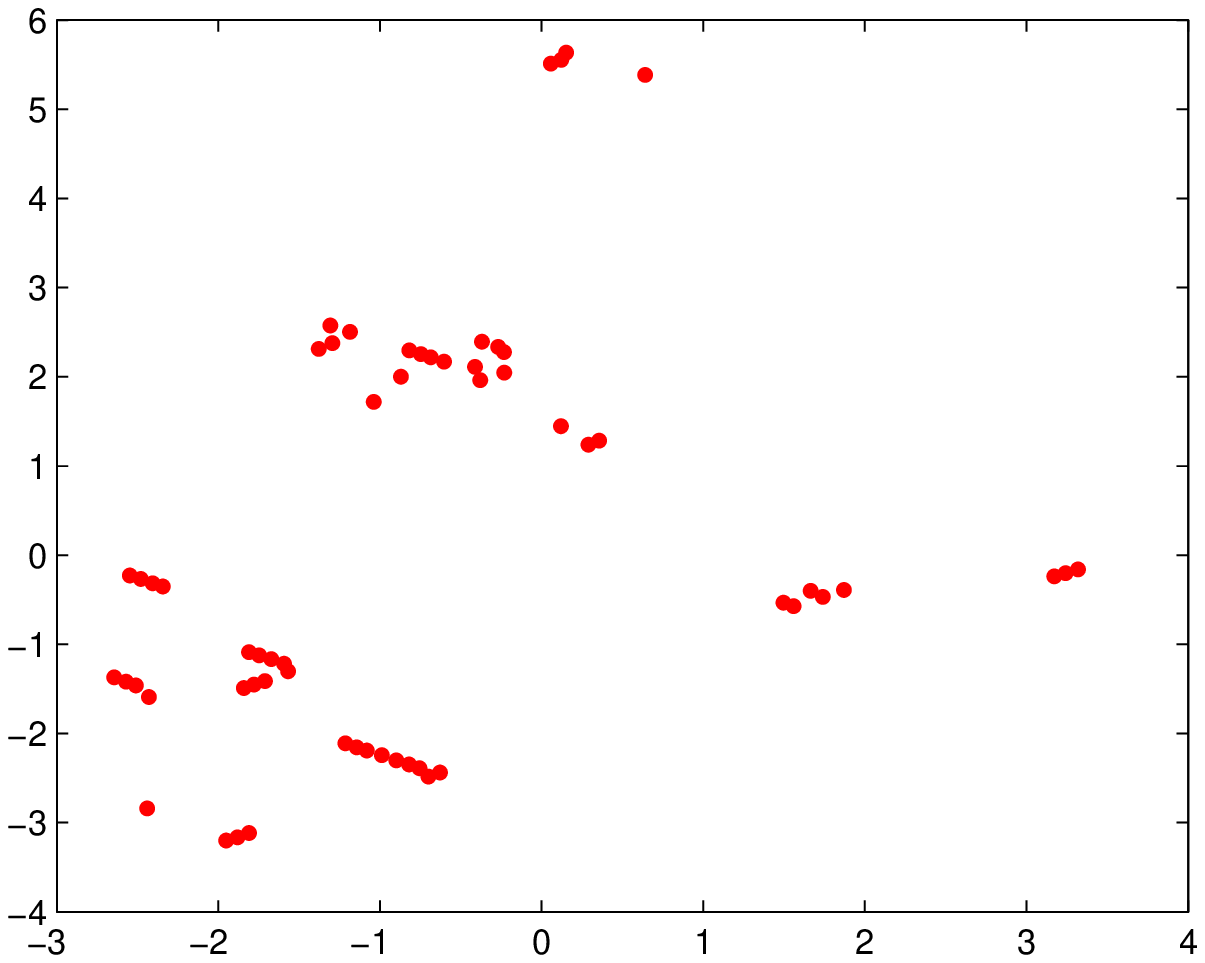}
\includegraphics[width=.49\columnwidth]{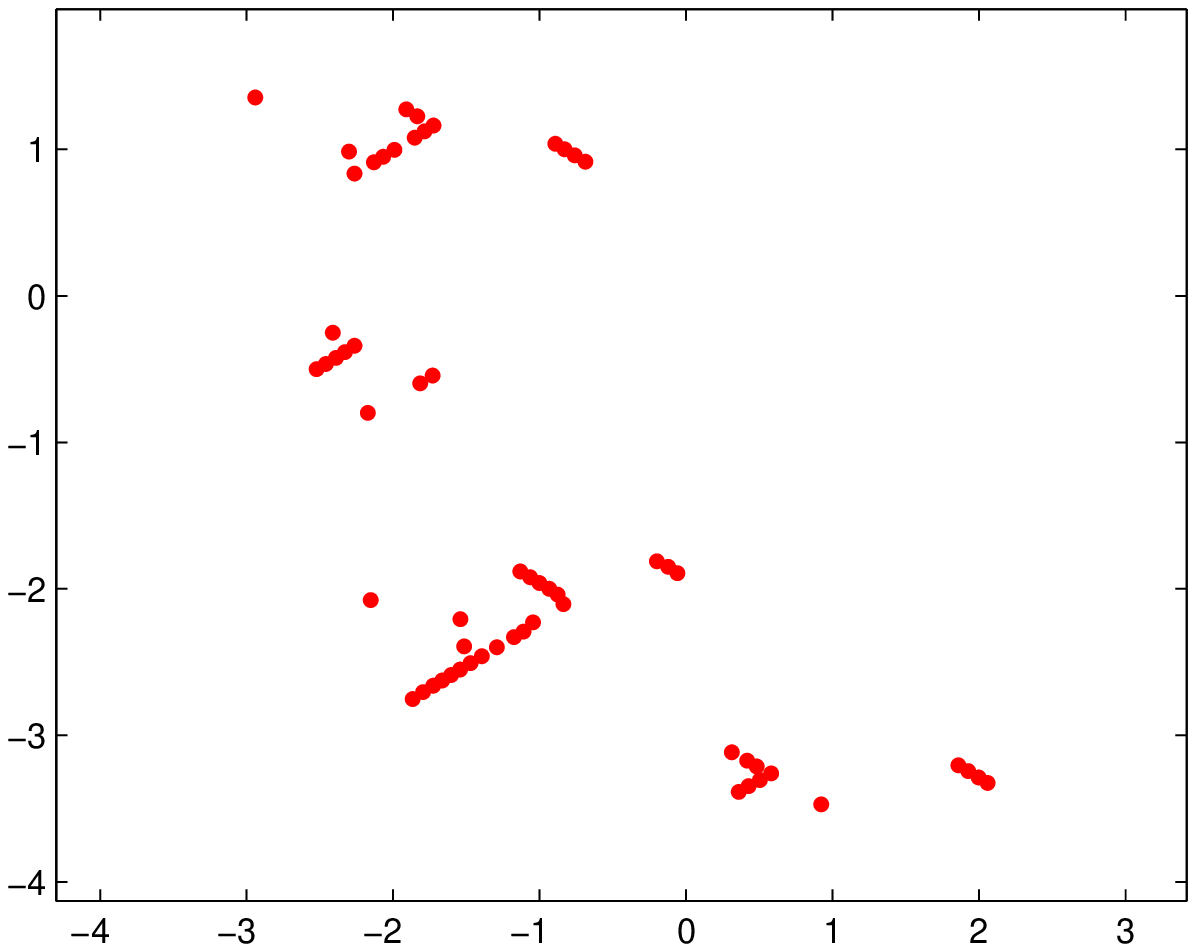}
\includegraphics[width=.49\columnwidth]{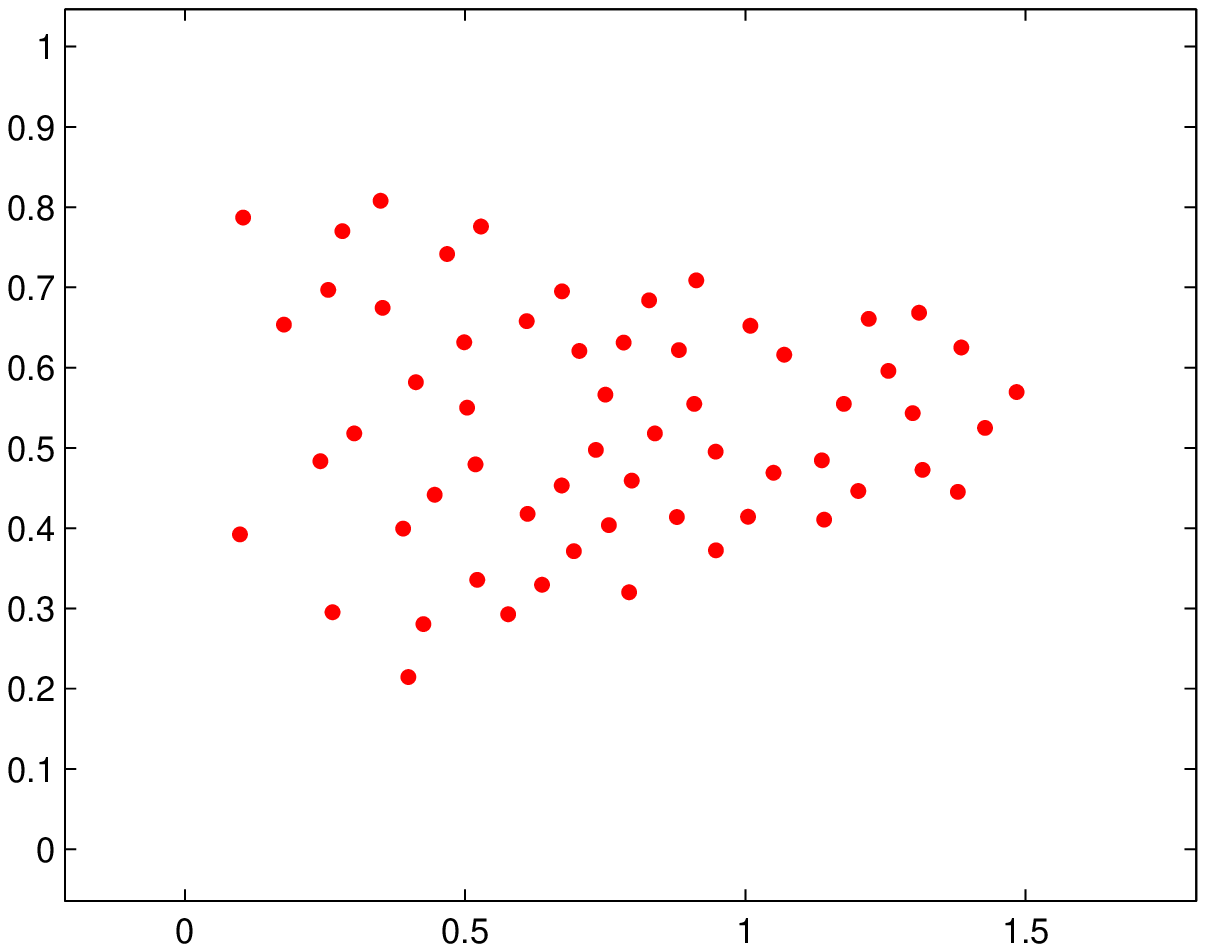}
  \caption{Configurations obtained with $n=1,3,7,N-1$, respectively, and $\alphar=60\degs$, $N=60$. The group is moving horizontally from left to right.}
  \label{fig:shots}
\end{figure}

%%%%%%%%%%%%%%%%%%%%%%%%%%%%%%%%%%%%%%%%%%%%%%%%%%%%%%%%%%%%%%%%%%%%%%%%%%%%%%%%%%%%%%%%%%%%%%%%%%%%%%%%%%
%%%%%%%%%%%%%%%%%%%%%%%%%%%%%%%%%%%%%%%%%%%%%%%%%%%%%%%%%%%%%%%%%%%%%%%%%%%%%%%%%%%%%%%%%%%%%%%%%%%%%%%%%%
%%%%%%%%%%%%%%%%%%%%%%%%%%%%%%%%%%%%%%%%%%%%%%%%%%%%%%%%%%%%%%%%%%%%%%%%%%%%%%%%%%%%%%%%%%%%%%%%%%%%%%%%%

\section{Analytical results}\label{sec:analysis}
In this section we develop a framework for the analysis of the presented model which helps the interpretation of the numerical results.
Several analytical tools have been developed and used in literature for the analysis of flocking algorithms.
The typical method for their analysis consists in defining a suitable potential function, called ``virtual'' or ``artificial'' potential, whose gradient gives the dynamics. Hence a well developed theory on potential systems can be used to make a full mathematical analysis, see e.g.~\cite{YL-RL-LEK:08,AM-LE-LB-AS:03}.
We have anticipated in the introduction that our model has two main features: {\em state-dependent switching}, and {\em asymmetry of interactions}.
Let us illustrate them. It is clear from their definition that the attraction and repulsion neighborhoods $\creg{}{n}$ and $\rreg{}{n}$ depend in a discontinuous way on the configuration (i.e. on the positions of the animals). As a consequence, the system's evolution switches among a finite collection of equations.
From a system-theoretical point of view, equation~\eqref{eq:system-adim} is a switching system~\cite{DL:03} with state-dependent switches. 
These issues have already been taken into account in flocking studies, as many papers consider the case of animals which interact when they are closer than a certain threshold. See, among others,~\cite{HGT-AJ-GJP:07}. The virtual potentials approach can be extended to these problems, provided the potential is allowed to be non-differentiable at the switching points. However, our model has the distinctive feature that, because of the limitation in the number of neighbors and in the shape of the sensitivity zones, interactions need not be symmetric. Example giving, for two animals $i,j$, the inclusion $j\in\creg{i}{n}$ does not imply that $i\in\creg{j}{n}$. This fact prevents us from using a virtual potentials approach, as the operation of differentiating the potential does not keep any directionality information. Actually, in literature there is no general approach available for systems with directed interactions. A partial asymmetry has been taken into account in other works~\cite{HS-LW-TC:06}, but the given treatment is far for being sufficient for our purposes.
In what follows we lay down the basics of a theory that we believe is able to catch the specific features of our model.

Let us start by defining solutions of a switching system in an useful sense. Indeed, a differential equation of the form ${\dot x=f(x)}$ with a discontinuous right-hand-side can not have a solution in the classical sense, i.e. a solution which is differentiable. Let us consider the right-hand side of equation~\eqref{eq:system-adim}.
That expression is not defined on the set of {\em degenerate configurations}, that is
\begin{equation*}
\CoincSet=\setdef{x\in (\real^2)^N}{\exists \,i,j \in \until{N} \text{ s.t. } x_i=x_j}.
\end{equation*}
Moreover, it is not always well defined on the set of {\em switching configurations}, in which two or more agents are equidistant from another,
\begin{equation*}
\DiscSet=\setdef{x\in (\real^2)^N}{\exists \,i,j,k \in \until{N} \text{ s.t. } \|x_i-x_j\|=\|x_i-x_k\|},
\end{equation*}
because of the ambiguity in the definition of the ``$n$ closest neighbors''.

To have a proper definition, we shall consider the following differential equation
\begin{equation}\label{eq:system-abstract}
\dot x(t)=f(x(t)),
\end{equation}
where the flow $\map{f}{(\real^2)^N\setminus (\DiscSet\union\CoincSet)}{(\real^2)^N}$ is defined componentwise as in equation~\eqref{eq:system-adim},
$$ f_i(x)=\sum_{j\in \creg{i}{n}} (x_j-x_i)  -  \xi^2 \sum_{j\in \rreg{i}{n}} \frac{(x_j-x_i)}{\|x_j-x_i\|^2}\qquad\forall i\in\until{N}.$$
Note that $f$ can not be extended with continuity to the set $\DiscSet\union\CoincSet$. Hence, a solution involving, for instance, two animals equidistant from a third one, can not be defined in the classical sense. In what follows, we shall extend the definition of the solutions of equation~\eqref{eq:system-abstract} to include the set $\DiscSet$. For such extension, we shall follow the approach in~\cite{AFF:88}, which requires defining a suitable {\em differential inclusion}, derived from~\eqref{eq:system-abstract}.

To this goal, let $B(y,\delta)$ denote the Euclidean ball of radius $\delta$, centered in $y$, and set
$$\K f(y)=  \bigintersect_{\delta>0}\bigintersect_{\mu(\Lambda)=0}\co\left\{f(B(y,\delta)\setminus \Lambda)\right\},$$
where the operator $\co$ denotes closed convex hull, and $\mu$ denotes Lebesgue measure.
The map $\map{x}{\realnonnegative}{(\real^2)^N}$ is said to be a {\em Filippov solution} of the system~\eqref{eq:system-abstract} if it is absolutely continuous and it satisfies the differential inclusion
$$\dot x(t)\in \K f(x(t)) $$
for almost any $t>0.$

From now on, we restrict ourselves to the case in which $\alphac=360\degs$, $\alphar=360\degs$, and $n=1$. Hence, each animal interacts just with its closest mate, and $\creg{i}{1}=\rreg{i}{1}$ for every $i$. We make this assumption because, while the analysis for this case is simpler than for the general one, still the significant features of switching and asymmetry are apparent. Indeed, provided $N>2$, the relation of ``being the closest to'' needs not to be symmetrical. Moreover, simulations show for this case the formation of regular structures, as in Figure~\ref{fig:cluster-views}, which can be of intrinsic interest.

In the case we are considering, it is useful to define the set of the closest neighbors of a given animal $i\in\until{N}$ as
$$\closest{i}(x)=\arg\min\limits_{j\neq i}\{\|x_i-x_j\|\}.$$
Notice that $\closest{i}(x)$ may be multivalued when $x\in\DiscSet$, and let $\card{\closest{i}(x)}$ denote its cardinality, that is the number of closest neighbors of animal $i$.

Given the above definitions, we are able to prove the existence of Filippov solutions.
\begin{theorem}[Existence]\label{th:existence} Let $n=1$. Then, for any initial condition $x^0\in(\real^2)^N\setminus\CoincSet$, equation~\eqref{eq:system-adim} has at least one Filippov solution $x$, such that $x(0)=x^0$ and $x(t)\in (\real^2)^N\setminus\CoincSet $ for every $t>0$.
\end{theorem}
\begin{proof}%\normalfont
We remark that $f$ is piecewise continuous in the following sense.
Let
\begin{equation}\label{eq:LabelSet}
\V=\setdef{v\in\until{N}^N}{\forall i\in\until{N},\, v_i\neq i}.
\end{equation}
For any $v\in\V$, let us also define the open, possibly empty, set
$$E_v=\setdef{x\in(\real^2)^N\setminus\CoincSet}{\forall i\in\until{N},\, \|x_{v_i}-x_i\|<\|x_{k}-x_i\|, k\notin\{i,v_i\} }$$
which is the set of the configurations such that $v_i$ is $i$'s closest neighbor.
Note that $E_v\intersect E_u=\emptyset$ if $u\neq v$ and the measure of the boundary of each $E_v$ is zero. Moreover, denoting by $\bar E_v$ the closure of $E_v$
in the induced topology of $(\real^2)^N\setminus\CoincSet$, it holds that $\bigcup_{v\in \V} \bar E_v=(\real^2)^N\setminus\CoincSet$.
In the interior of each ``piece'' $E_v$, that is if $x\in E_v$, we have that $f(x)=f^v(x)$, where the function $f^v$ is defined componentwise as
$$f^v_i(x)=\left(x_{v_i}-x_i\right)-\xi^2 \frac{x_{v_i}-x_i}{\|x_{v_i}-x_i\|^2}.$$
Moreover, each function $f^v$ is continuous in $\bar E_v$.
These facts imply that for every $x\in (\real^2)^N\setminus\CoincSet$, the set $\K f(x)$ is bounded, nonempty, closed and convex, and the map $x \mapsto \K f(x)$ is upper semicontinuous.
Before we conclude, we need to show that, provided $x(0)\in (\real^2)^N\setminus\CoincSet$, the solution can not reach $\CoincSet$ either in finite time or asymptotically.
By contradiction, let
\begin{equation}\label{eq:min}\lim_{t\to L} \min_{h,k\in\until{N}}{||x_h(t)-x_k(t)||}=0,
\end{equation}
for $L\in (0,\infty].$
Note that for each $t\in[0,L)$, there exist a pair $(h^*,k^*)\in \until{N}^2$, possibly depending on time, such that it attains the minimum in~\eqref{eq:min}, and thus ${h^*}=\closest{{k^*}}(x(t))$ and ${k^*}=\closest{{h^*}}(x(t))$.
By continuity of the solution, it exists $t_0\in(0,L)$ and $\eps\in (0,\xi)$ such that for $t\in [t_0,L)$, we have $||x_{h^*}(t)-x_{k^*}(t)||\le \eps<\xi$. This implies that repulsion is larger than attraction, and then the animals are moving away from each other, that is $||x_{h^*}(t)-x_{k^*}(t)||$ is increasing in $t$. This contradicts equation~\eqref{eq:min}.

At this point, we can apply~\cite[\S7~Theorem~1]{AFF:88}, %\cite[Theorem~1.3]{AB-LR:05}
and obtain that the differential inclusion $\dot x(t)\in \K f(x(t))$ has at least one solution $x(t)$, for all $t>0$ and for any initial condition $x^0\in(\real^2)^N\setminus\CoincSet.$
\end{proof}
\begin{remark}
The above proof can be extended to the case $n>1$ and ${\alphac<360\degs}$, ${\alphar<360\degs}$, modulo a suitable redefinition of the ``pieces'' $E_v$, in order to account for the more complex neighborhood relationships: the notational setup would be cumbersome, and we do not detail it.
\end{remark}

\medskip
Loosely speaking, we would expect that a configuration having the closest neighbor(s) at distance $\xi$ for all animals, as in the lattice configuration of Figure~\ref{fig:cluster-1} would be an equilibrium configuration. Indeed, each animal $i$ is driven by the attraction-repulsion force component $f_i$ towards keeping a distance $\xi$ from its neighbors. The following result technically clarifies this intuition. Before the statement, we need to define a configuration $x^*$ to be a {\em Filippov equilibrium} of $f$ when $0\in K\,f(x^*)$.
\begin{proposition}[Equilibria]\label{prop:6-equilibria}
Let $x^*\in(\real^2)^N\setminus\CoincSet.$
If for all $i\in\until{N}$ and for all $k\in \closest{i}(x^*)$, it holds $\|x^*_i-x^*_k\|=\xi$,
then $x^*$ is a Filippov equilibrium for the system~\eqref{eq:system-adim}, and moreover $1\le\card{\closest{i}(x^*)}\le 6$, for all $i\in \until{N}.$
\end{proposition}
\begin{proof}
We consider $\closest{i}(x^*)$ for every $i\in \until{N}$, and distinguish the cases in which their cardinalities are equal to, or larger than 1. If $\card{\closest{i}(x^*)}=1$ for all $i\in\until{N}$, then $f$ is smooth at $x^*$, and hence $\K f(x^*)=\{f(x^*)\}=\{0\}$.
If instead there exists $h\in\until{N}$ such that $\card{\closest{h}(x^*)}>1$, then let us consider
the set $\V$ defined in~\eqref{eq:LabelSet}, and take $v\in \V$ such that $v_i\in\closest{i}(x^*),$ for every $i\in\until{N}$. Since $x^*$ is an accumulation point of $E_v$, let us consider a sequence $\{x^l\}_{l\in \natural}\subset E_v$, such that $x^l\to x^*$ as $l\to+\infty$. Hence, $f(x^l)\to 0$ as $l\to\infty$, and this implies, by the definition of the differential inclusion, that $0\in \K f(x^*).$

By definition, $\card{\closest{i}(x^*)}\ge 1$. The fact that $\card{\closest{i}(x^*)}\le 6$ can be shown by contradiction. Let $\card{\closest{i}(x^*)}=m$, with $m>6$, for some $i$. Then there are $m$ points of $\real^2$, representing positions, which belong to a circle of radius $\xi$ centered in $x^*_i$. But then the distance between two of them has to be less than $\xi$, which is a contradiction.
\end{proof}

Hexagonal lattice configurations, in which animals are at distance $\xi$ from their closest neighbor(s), are observed in simulations in the case $\alphac=\alphar=360\degs$ and $n=1$, as reported in Figure~\ref{fig:cluster-1}. Indeed, Proposition~\ref{prop:6-equilibria} shows that such a lattice is an equilibrium: loosely speaking, we can say that it is the most closely packed configuration among the equilibria pointed out by this result. Notice moreover that such lattice equilibria are actually switching configurations, belonging to the set $\DiscSet$: this gives an a posteriori justification of the effort that we have done for a careful extension of the solutions to this set.

%%%%%%%%%%%%%%%%%%%%%%%%%%%%%%%%%%%%%%%%%%%%%%%%%%%%%%%%%%%%%%%%%%%%%%%%%%%%%%%%%%%%%%%%%%%%%%%%%%%%%%%%%%%%
%%%%%%%%%%%%%%%%%%%%%%%%%%%%%%%%%%%%%%%%%%%%%%%%%%%%%%%%%%%%%%%%%%%%%%%%%%%%%%%%%%%%%%%%%%%%%%%%%%%%%%%%%%
%%%%%%%%%%%%%%%%%%%%%%%%%%%%%%%%%%%%%%%%%%%%%%%%%%%%%%%%%%%%%%%%%%%%%%%%%%%%%%%%%%%%%%%%%%%%%%%%%%%%%%%%%%%
%
%
%
%
\section{Discussion}\label{sec:discussion}
Our agent-based model has been conceived from assumptions which are widely accepted from the biological point of view, and it shows the onset of group structures and patterns which are observed in nature. In some sense, our model can be seen as a generalization of other similar models, since in the special case $(\alphar, \alphac)=(360\degs,360\degs)$ we recover results which are by now consolidated in literature. The novelty resides in that by modifying a small set of key parameters, we obtain other patterns, which are experimentally observed in animal groups. Our results suggest that apparently large differences in group patterns may arise just from differences in the attraction and repulsion sensitivity zones. In this perspective, we would like to point out that the group structure is not only a function of the species, but also of the external conditions. For instance, surf scoters~\cite{RL-LEK:09} and other animals can form either clusters or lines, depending on the environmental conditions. This suggests that $\alphar$ and $\alphac$ are not only animal-dependent, but they can also vary depending on the type of motion, environmental conditions, presence of predators, and aim of the displacement.
In this paper, the choices of the angles were made to explore the model's properties and they do not come from experimental data. Nevertheless, the correspondence of the simulated patters with those experimentally observed suggests that future research should distinguish between the (variable) sensitivity zones and the (fixed) visual field, as in~\cite{SG-SAL-DIR:96,AN-VCB:08}.

Now we review the results obtained in the previous sections. The goal is to propose some biological insights, especially about the role of anisotropy in the interactions.
As we did presenting the simulations, we distinguish three cases: isotropic attraction and repulsion; anisotropic attraction and repulsion; isotropic attraction and anisotropic repulsion.

In the first case, with completely isotropic interactions, our model produces clusters of individuals. Clusters are common for small birds (e.g., starlings~\cite{AC-IG:08e}) and for fish, whose sensing abilities allow an almost complete perception around them. Moreover, these animals often change the direction of motion following complex trajectories, and this makes useful to keep under control all the space around them. By means of isotropic interactions, they are able to change rapidly the direction of motion of the whole group, with no need to significantly rearrange the shape of the group. This would not be possible in less symmetric formations like lines or Vees.

Instead, when interactions are not isotropic, also the obtained formations are not isotropic, and we observe that their onset depends on spontaneous leader-following mechanisms.
When both repulsion and attraction are focused in front, the leader-following mechanism induces the formation of lines. Lines are commonly observed in slow-moving animals as lobsters, elephants and penguins. The small repulsion angle $\alphar$ appears to be crucial to obtain such a pattern (see Figure~\ref{fig:elongation}). This could be related to the fact that such animals keep a steady direction of motion, and, once the line is established, repulsion needs to be active only against the forerunners.
The choice $\alphac=180\degs$ can also be significant from the biological point of view. Indeed, this would mean that individuals pay more attention to the group mates in front, while they do not respond to what happens behind them. For example, they would not perceive a disconnection of the group. Hence, a model assuming $\alphac=180\degs$ seems suitable either for animals with a restricted frontal visual field or for animals which are not particularly interested in the cohesion of the whole group. %``selfish'' animals which prefer to follow the informed individuals at the head of the group rather than paying attention to the cohesion of the whole group.

A different leader-following mechanism induces V-like formations when attraction is isotropic and repulsion is restricted to the group mates in front. Our simulations reproduce several significant features of natural V-like formations, described in the experimental literature~\cite{LLG-FHH:74,FHH:74,FRH:87}. First, echelons and J-like formations are as common as perfect Vees. In our model all these formations appear, and actually the behavior of single individuals does not depend on the global shape of the formation. Second, V-like formations are not stable, but rather they often disband and quickly reform.  Third, in our model each Vee is made of a limited number of individuals (see Figure~\ref{fig:Vee}), independently of the total number of animals $N$. In other words, increasing $N$ leads to the formation of a larger number of V-like groups, but not to larger groups.
This fact has been experimentally observed~\cite{FRH:87} and explained~\cite{PS-AP-JKH:03} resorting to the argument of string instability, that we introduced in Section~\ref{subsec:lines}. Our results should be compared to those in~\cite{AN-VCB:08}, where authors propose an ad-hoc formation algorithm based on aerodynamic arguments in which the number of V-like groups is constant for increasing $N$.

The function of V-like formation has been greatly investigated (see for instance~\cite{FRH:87,CC-JS:94,JRS-DB:98,HW-JM-YC-PA-SJ:01,PS-AP-JKH:03} and references therein). Two hypothesis are the most considered: aerodynamic advantage, and visual communication advantage.
The former is based on the fact that each flying individual creates an upwash region behind it, just off the tips
of its wings, so that another individual can benefit placing itself in that region.
The latter is instead based on the fact that flying in a skewed position with respect to the bird in front is useful to avoid collisions, and allows an unhindered visual communication with all the group mates~\cite{CC-JS:94}.
Our results appear to support the hypothesis of visual communication advantage, since V-like formations are not obtained imposing individuals to stay in upwash regions. Instead, they are obtained just from frontal repulsion and, for the sake of group cohesion, isotropic attraction.

Motivated by some literature~\cite{IA:80,AH-CW:92,AC-IG:08a}, we have also investigated the dependence of the configurations on $n$, the number of interacting neighbors.
We have seen that cluster configurations can be obtained with any $n$, but the value of $n$ influences the internal structure of the cluster. Moreover, it appears that an intermediate number of neighbors is preferable to obtain lines or V-like formations. These results seem to confirm the reasonableness of topological interactions with a limited number of neighbors. However, this point deserves further investigations.

%%%%%%%%%%%%%%%%%%%%%%%%%%%%%%%%%%%%%%%%%%%%%%%%%%%%%%%%%%%%%%%%%%%%%%%%%%%%%%%%%%%%%%%%%%%%%%%%%%%%%%%%%%%%%

\subsection*{Future research}
We are keen on developing our research in three respects. First, it should be noted that although the presented model is bidimensional, it can readily be extended to a three-dimensional one. We are planning to study a 3D version of our model since the novel set of simulations might show interesting phenomena. We are especially interested in studying the effect of the anisotropy in 3D V-like formations, because the skewed formations described in~\cite{FHH:74} are inherently three-dimensional.

Second, we want to further develop the theoretical analysis of the model, in order to include a thorough description of the equilibria of~\eqref{eq:system-simulated}, and their stability analysis in the suitable switching systems framework. Moreover, we would be interested in a variational interpretation of the model. Indeed, while it is apparent that each animal is trying to minimize a ``private'' potential depending on its neighbors, it is unclear whether and when this would result in a configuration minimizing some global objective function.

Third, we believe that the simplicity of the interaction rules we have proposed, and the limited number of group mates to be kept into account, can be interesting features for engineering applications. In particular, we will investigate the design of robust algorithms for environmental deployment of robots or sensors and the formation cruising of unmanned vehicles.

\bibliographystyle{plain}
%\bibliography{aliasFrasca,RefFrasca,PF}

\end{document}